\documentclass{siamltex1213}

\RequirePackage{amsmath,amsfonts}
\usepackage{graphics}
\usepackage{graphicx}

\usepackage{latexsym}
\usepackage{multirow}
\RequirePackage[numbers]{natbib}
\usepackage{mathtools}
\allowdisplaybreaks[1]
\usepackage{color}
\newtheorem{thm}{Theorem}[section]

\newtheorem{assump}{Assumption}[section]
\newtheorem{rmk}{Remark}[section]
\newtheorem{lem}{Lemma}[section]
\newtheorem{cor}{Corollary}[section]
\newtheorem{algo}{Algorithm}[section]

\newcommand{\CEsp}[3]{ \textbf{E}^{#1}_{#3} \left[  #2  \right] }
\newcommand{\norm}[1]{\left| #1 \right|}

\newcommand{\Done}[2]{  {\frac{ \partial  #1}{\partial #2}}  }
\newcommand{\Dtone}[2] { \frac{ \partial^2 #1}{\partial #2^2} }

\newcommand{\revised}[1]{\textcolor{black}{#1}}

\title{A CONVOLUTION METHOD FOR NUMERICAL SOLUTION OF BACKWARD STOCHASTIC DIFFERENTIAL EQUATIONS}

\author{Cody B. Hyndman\footnotemark[2] \footnotemark[3] 
\and Polynice Oyono~Ngou\footnotemark[4] \footnotemark[5]}

\begin{document}
\maketitle

\renewcommand{\thefootnote}{\fnsymbol{footnote}}

\footnotetext[2]{Corresponding author. Concordia University, Department of Mathematics and Statistics, 1455 boulevard de Maisonneuve Ouest, Montr\'eal, Qu\'ebec, Canada H3G 1M8 (cody.hyndman@concordia.ca)}
\footnotetext[3]{This research was supported by the Natural Sciences and Engineering Research Council of Canada (NSERC).}
\footnotetext[4]{Concordia University, Department of Mathematics and Statistics, 1455 boulevard de Maisonneuve Ouest, Montr\'eal, Qu\'ebec, Canada H3G 1M8}
\footnotetext[5]{An earlier version of paper was presented at the Oxford-Man Institute of Quantitative Finance's   \textit{Young Researchers Meeting on BSDEs, Numerics and Finance} in July 2012.  The helpful comments of the meeting participants are gratefully acknowledged.}

\renewcommand{\thefootnote}{\arabic{footnote}}

\pagestyle{myheadings}
\thispagestyle{plain}
\markboth{C.~B. HYNDMAN AND P. OYONO~NGOU \quad \quad \quad Convolution method for BSDEs}{C.~B. HYNDMAN AND P. OYONO~NGOU \quad \quad \quad Convolution method for BSDEs}

\begin{abstract} 
We propose a new method for the numerical solution of backward stochastic differential equations (BSDEs) which finds its roots in Fourier analysis. The method consists of an Euler time discretization of the BSDE with certain conditional expectations expressed in terms of Fourier transforms and computed using the fast Fourier transform (FFT).  The problem of error control is addressed and a local error analysis is provided.   We consider the extension of the method to forward-backward stochastic differential equations (FBSDEs) and reflected FBSDEs.  Numerical examples are considered from finance demonstrating the performance of the method.
\end{abstract}


\noindent
\textbf{AMS 200 subject classifications:} Primary 60H10, 65C30; secondary 60H30.

\section{Introduction}  Backward stochastic differential equations (BSDEs) have been a topic of interest since the early work of \cite{bismut:1973} and the results of \cite{pardouxpeng:1990} on their well-posedness. A BSDE is an equation of the form 
\begin{equation} Y_t = \xi + \int_t^T f(s,Y_s,Z_s)ds - \int_t^T Z_s^{*} dW_s \label{eq:bsde} \end{equation} defined on a complete filtered probability space $\left( \Omega , {\bf P}, \mathcal{F}, \{ \mathcal{F}_t\}_{t \in [0,T]} \right)$ where $W$ is a standard $n$-dimensional Brownian motion, the terminal condition $\xi \in \mathbb{R}^k$ is a square integrable $\mathcal{F}_T$-measurable random variable and the driver $f:[0,T]\times\mathbb{R}^k \times \mathbb{R}^{k \times n} \rightarrow \mathbb{R}^k$ is a functional. It is known from \cite{pardouxpeng:1990} that there exists a unique adapted square integrable backward process $Y$ taking values in $\mathbb{R}^k$ and a unique predictable process $Z$ with values in $\mathbb{R}^{n \times k}$ satisfying equation (\ref{eq:bsde}) under Lipschitz and integrability conditions on the driver $f$. 

Many works have extended this existence and uniqueness result. \cite{antonelli:1993} introduced forward-backward stochastic differential equations (FBSDEs). \cite{mao:1995}, \cite{lepeltiersm:1997} and \cite{kobylanski:2000} among others treat non-Lipschitz cases. Also, the theory of BSDEs has found various applications, particularly in finance and in the study of partial differential equations (PDEs). 
From \cite{MR1176785} (see also \cite[Section 4.1]{elkaouri:1996}) we have
that if the Cauchy problem to the one-dimensional diffusion PDE 
\begin{equation} \begin{cases} \Done{u}{t} + \frac{1}{2}  \Dtone{u}{x} + f(t,u,\Done{ u}{x} ) = 0 \text{, } (t,x) \in [0,T) \times \mathbb{R}\\
u(T,x) = g(x) \end{cases} \label{eq:pdebsde} 
\end{equation}
has a unique solution $u \in \mathcal{C}^{1,2}$ then the solution $(Y,Z)$ for the \revised{one-dimensional} BSDE (\ref{eq:bsde}) with terminal condition $\xi = g(W_{T})$ admits the representation 
\begin{eqnarray}  Y_t & = & u(t,W_t) \label{eq:soly}\\
Z_t & = &  \Done{u}{x}(t, W_t). \label{eq:solz}\end{eqnarray}
Conversely, the solution of the PDE~(\ref{eq:pdebsde}) can be interpreted in terms of the solution of the BSDE~(\ref{eq:bsde}).
General formulations of the nonlinear Feynman-Kac formula for FBSDEs, quasilinear parabolic PDEs, and viscosity solutions have been studied extensively.

Deriving an explicit solution to a nontrivial (F)BSDE is possible only in very few situations, such as \cite{yong:1999}, \cite{hyndman:2009} and \cite{richter:2012}. Thus, numerical methods for BSDEs have been studied extensively.  Numerical methods for (F)BSDEs can be classified into three main groups: PDE based methods, spatial discretization based methods, and Monte-Carlo based methods.  PDE based methods, which started with the finite difference approach of \cite{douglas:1996}, consider a numerical resolution of the nonlinear parabolic PDE related to the (F)BSDE. The two other methods rely on a time discretization of the (F)BSDE.  Spatial discretization based methods (see \cite{chevance:1997}, \cite{ballypages:2003}, \cite{delaruem:2006}, \cite{crisanm:2012}, \cite{RuijOosl:2013} or \cite{pengxu:2011} among others) use a deterministic space grid. On the other hand, the space discretization is random in Monte-Carlo based methods (for instance, \cite{bouchardtouzi:2004},  \cite{gobet:2005}, and \cite{benderdenk:2007}).

In this paper, we propose an alternative spatial discretization method for BSDEs and illustrate its implementation in the one-dimensional case. To the best of our knowledge, the most efficient approach\revised{, in terms of speed and accuracy,} in this simple case is the binomial method of \cite{pengxu:2011} which has connections with the theoretical work of \cite{briand:2001} and \cite{maprottermt:2002}. However, our method avoids a notable drawback of the binomial method: the contraction of the space grid leading to the approximation of the Wiener process by means of scaled random walks. Indeed, we use a fixed equidistant space grid, thus allowing an exact simulation of the Wiener process at time nodes. The FFT algorithm, which plays a key role in our method, helps in producing an efficient algorithm. As in \cite{carrmadan:1999} and \cite{lordal:2008} in the context of option pricing under L\'evy processes, we employ the FFT algorithm to compute quadratures. The presence of dynamic programming through the Euler scheme is a major similarity between our method and \cite{lordal:2008}. The method presented in \cite{RuijOosl:2013} is somewhat similar in that transform methods are employed, but differs in the discretization scheme and the use of the Fourier-cosine expansion.

This paper is structured as follows. Section \ref{sec:timedisc} reviews Euler time discretization schemes for BSDEs which are used in Section \ref{sec:conv} to develop the convolution method. Section \ref{sec:error} presents a detailed error analysis of the convolution method. Some extensions of the method are presented in Section \ref{sec:extensions}, numerical results for examples from finance are included in Section \ref{sec:numres}, and Section \ref{sec:concl} concludes.

\section{Time discretization of BSDEs} \label{sec:timedisc}
The convolution method developed in this paper, as any spatial discretization based method, requires the availability of a time discretization scheme for BSDEs. In this section, we \revised{consider well-known} Euler time discretization schemes that are widely used in numerical methods for BSDEs.  Alternatives to the Euler schemes can be found in the $\theta-$schemes of \cite{Zhao2:2006} or the penalization scheme proposed by \cite{pengxu:2011} inspired by a method used by \cite{elkarouial:1997} to prove well-posedness of reflected BSDEs. Convergence of the Euler schemes are considered by \cite{zhang:2001, zhang:2004} and \cite{bouchardtouzi:2004}.

For simplicity of notation we shall suppose all processes are one-dimensional %
($k=n=1$).  Further, we make the following assumption to ensure existence and uniqueness of a solution to the BSDE~(\ref{eq:bsde}).
\begin{assump} We suppose a Markovian terminal condition with \begin{equation} \xi = g(W_T)\end{equation}  where $g : \mathbb{R} \rightarrow \mathbb{R}$ is real function satisfying the square integrability condition 
\begin{equation} \CEsp{}{ \xi^2 }{} = \CEsp{}{g(W_T)^2}{} < \infty. \end{equation}  
In addition, both the terminal condition $g$ and the driver $f$ verify the Lipschitz condition
\begin{align} 
\norm{g(x) - g(\bar{x})} + \norm{ f(t,y,z) - f(\bar{t},\bar{y},\bar{z})}  < C \left( \norm{x - \bar{x}} + \norm{y- \bar{y}} + \norm{z - \bar{z}} \right)   \label{eq:lips} 
\end{align} 
for some constant $C > 0$, $\forall x,\bar{x},y,\bar{y},z,\bar{z} \in \mathbb{R}$, and $\forall t,\bar{t} \geq 0 $.
\label{assump:bsde}
\end{assump}

Consider the time mesh  $\pi := \{ 0=t_0 < t_1 < ... <t_n=T \}$ on the time interval $[0,T]$ with $n \in \mathbb{N}$ time steps. \revised{Define $\Delta_i := t_{i+1} - t_i$ and $\Delta W_i = W_{t_{i+1}} - W_{t_i}$  for $i=1,\ldots,n$.  Write  $\norm{\pi} := \max_{0 \leq i < n} \norm{t_{i+1} - t_i}$ for the maximal time step. } Let ${Y}^{\pi}_{t_i}$ and ${Z}^{\pi}_{t_i}$ denote the approximate solution at time node $t_i$ \revised{and $\xi^{\pi}$  be an approximation of the terminal condition $\xi$}.  \revised{For} BSDEs with a Markovian terminal condition we set $\xi^{\pi} = \xi = g(W_T)$.

\revised{A well known method  (see, for example, \cite{zhang:2001,zhang:2004}, \cite{bouchardtouzi:2004}, or \cite{bouchardet:2009},  for a standard derivation)  for obtaining an approximate solution based on the Euler time-discretization of equation~(\ref{eq:bsde}) is given by the backward recursion}
\begin{equation}\begin{cases} {Z}^{\pi}_{t_n} = 0 \text{, } {Y}^{\pi}_{t_n} = \xi^{\pi} \\
{Z}^{\pi}_{t_i} = \frac{1}{\Delta_i} \CEsp{}{ {Y}^{\pi}_{t_{i+1}} \Delta W_i | \mathcal{F}_{t_i} }{} \text{, } 0 \leq i < n\\  {Y}^{\pi}_{t_i}=\CEsp{}{{Y}^{\pi}_{t_{i+1}}+f({t_i},{Y}^{\pi}_{t_{i+1}},{Z}^{\pi}_{t_i})\Delta_i | \mathcal{F}_{t_i}}{} \text{, } 0 \leq i < n \end{cases} \label{eq:expdisc1}\end{equation}
which we call the explicit Euler scheme I.  
Another explicit scheme consists of replacing the conditional expectation of the driver in the explicit Euler scheme I by the driver evaluated at the conditional expectations of the arguments.  This procedure leads to 
\begin{equation}
\begin{cases} {Z}^{\pi}_{t_n} = 0 \text{, } {Y}^{\pi}_{t_n} = \xi^{\pi} \\
{Z}^{\pi}_{t_i} = \frac{1}{\Delta_i} \CEsp{}{ {Y}^{\pi}_{t_{i+1}} \Delta W_i | \mathcal{F}_{t_i} }{} \text{, } 0 \leq i < n \\  {Y}^{\pi}_{t_i}=\CEsp{}{{Y}^{\pi}_{t_{i+1}}| \mathcal{F}_{t_i}}{}  + f({t_i},\CEsp{}{{Y}^{\pi}_{t_{i+1}}| \mathcal{F}_{t_i}}{},{Z}^{\pi}_{t_i})\Delta_i \text{, } 0 \leq i < n \end{cases} \label{eq:expdisc2}
\end{equation} 
which we call the explicit Euler scheme II.
The approximate $(Y,Z)$ processes then takes the form 
\begin{equation} 
{Y}^{\pi}_t = {Y}^{\pi}_{t_i} \text{, } {Z}^{\pi}_t = {Z}^{\pi}_{t_i} \text{ for } t \in [t_i, t_{i+1}).
\end{equation}
on the entire time interval.

The global discretization error $E_{\pi}$  is defined as 
\begin{equation} 
E_{\pi}^2 := \max_{0 \leq i <n-1} \CEsp{}{ \sup_{t \in [t_{i},t_{i+1}]} \norm{Y_t - {Y}^{\pi}_{t_i} }^2  }{} + \sum_{i=0}^{n-1} \CEsp{}{  \int_{t_i}^{t_{i+1}} \norm{Z_s - Z^{\pi}_{t_i}}^2 ds}{} 
\end{equation}
for any version of the Euler scheme.  Due to the Lipschitz nature of the driver $f$, it can be proved that the explicit schemes have a first order quadratic error as noted by \cite[Remark 2.1.1]{bouchardet:2009}.

\begin{thm} Under the setting of Assumption \ref{assump:bsde}, the Euler schemes yield a first order quadratic error
\begin{equation} 
E_{\pi}^2 =  \mathcal{O}(\norm{\pi}).
\end{equation}
\end{thm}
\revised{In the next section we formulate the convolution method.}

\section{Convolution method} \label{sec:conv}
In this section, we introduce the convolution method for the numerical solution of the BSDE~(\ref{eq:bsde}). The method involves expressing the conditional expectations in an explicit Euler time discretization of the BSDE as convolutions, calculating the Fourier transform of the approximate solution, applying the convolution theorem of Fourier analysis, and taking the inverse Fourier transform of the results in order to recover expressions for the approximate solution which are recursive backward in time.  In order to implement the convolution method we present the discretization of intermediate quadratures and their relationship to the discrete Fourier transform (DFT) which can be efficiently computed using the fast Fourier transform (FFT).

\subsection{Convolution on the explicit Euler scheme II} 

The starting point of the convolution method for BSDEs is an explicit Euler scheme. If we consider the explicit Euler scheme II of equation (\ref{eq:expdisc2}) an approximate solution of the BSDE (\ref{eq:bsde}) at mesh time $t_{i}$ consists of real-valued functions $u_{i}$, $\dot{u}_{i}$, and $\tilde{u}_{i}$ defined by the backward recursions
\begin{align} 
  u_i(x)  & =   \tilde{u}_i(x) + \Delta_i f(t_i, \tilde{u}_i(x) , \dot{u}_i(x)) \label{eq:uf} \\
  \dot{u}_i (x) & =  \frac{1}{\Delta_i} \CEsp{}{ u_{i+1}(W_{t_{i+1}}) \Delta W_i | W_{t_i} = x }{} \nonumber \\
  & =  \frac{1}{\Delta_i} \int_{-\infty}^{\infty} (y - x)u_{i+1}(y) h(y - x) dy \label{eq:guf} \\
  \tilde{u}_i(x) & =  \CEsp{}{ u_{i+1}(W_{t_{i+1}}) | W_{t_i} = x }{} \nonumber \\
  & =   \int_{-\infty}^{\infty} u_{i+1}(y) h(y - x) dy \label{eq:iuf} 
\end{align}
for $i=0,1,...,n-1$ and $u_n(x) = g(x)$. 
Note that $u_{i}$ represents the approximate $Y$ process and $\dot{u}_{i}$ stands for the approximate $Z$ process at mesh time $t_{i}$ while $\tilde{u}_{i}$ is an intermediate quantity.  The notation $u$ used in equations~(\ref{eq:uf})-(\ref{eq:iuf}) is not to be confused with the solution of the PDE~(\ref{eq:pdebsde}) as we do not employ the representation~(\ref{eq:soly})-(\ref{eq:solz}) in this paper.  The function $h$ is the density function of $W_{t_{i+1}}$ conditional on the value of $W_{t_{i}}$ 
\begin{equation} 
h(x) = (2 \pi \Delta_i)^{-\frac{1}{2}} \exp \left( - \frac{ x^2}{ 2 \Delta_i}  \right).
\end{equation} 

If a method for calculating the integrals of equations (\ref{eq:guf}) and (\ref{eq:iuf}) is available, then the sequence $\left( u_i(W_{t_i}) , \dot{u}_{i}(W_{t_i})\right)$ for $i=0,1,2,...,n-1$ is an approximation to the BSDE solution of equations (\ref{eq:soly}) and (\ref{eq:solz}) on the interval $[0,T]$.  The stationarity and independence of Brownian increments allow us, as in \cite{lordal:2008}, to express the functions $\tilde{u}_{i}$ and $\dot{u}_{i}$ in equations (\ref{eq:guf}) and (\ref{eq:iuf}) as convolutions. These convolutions suggest using Fourier transforms and hence the computation of the integrals via discrete Fourier transforms.

Recall the Fourier transform of an integrable real function $\eta$ is the function $\hat{\eta} : \mathbb{R} \rightarrow \mathbb{C}$ is defined as 
\begin{equation} \hat{\eta}( \nu ) := \mathfrak{F}[\eta](\nu) = \int_{-\infty}^{\infty} e^{-{\bf i} \nu x} \eta(x) dx \end{equation}
where ${\bf i} = \sqrt{-1}$ is the imaginary unit. The inverse Fourier transform recovers the function $\eta$ from its Fourier transform $\hat{\eta}$ through the relation
\begin{equation} \eta(x) := \mathfrak{F}^{-1}[ \hat{\eta} ](x) = \frac{1}{2 \pi} \int_{-\infty}^{\infty} e^{ {\bf i} x\nu } \hat{\eta}(\nu) d\nu. \end{equation} 
For any real function $\eta: \mathbb{R} \rightarrow \mathbb{R} $ define the dampened function $\eta^{\alpha}$ as
\begin{align} 
\eta^{\alpha}(x) & =  e^{-\alpha x} \eta(x). \label{eq:dampening}
\end{align} 
where $\alpha\in\mathbb{R}$ is a dampening parameter. Taking the Fourier transform of $\tilde{u}_{i}^{\alpha}$ in equation~(\ref{eq:iuf}) gives 
\begin{align} \mathfrak{F}[ \tilde{ u}^{\alpha}_i ](\nu) & =  \int_{-\infty}^{\infty} e^{ -{\bf i} \nu x} e^{-\alpha x} \int_{-\infty}^{\infty} u_{i+1}(y) h(y - x) dy dx  \nonumber \\
& = \int_{-\infty}^{\infty} e^{ -{\bf i} \nu x} \int_{-\infty}^{\infty} u^{\alpha}_{i+1}(y) e^{ \alpha(y-x) }h(y - x) dy dx \nonumber \\
& =  \mathfrak{F}[u^{\alpha}_{i+1}](\nu) \mathfrak{F}[e^{-\alpha z}h(-z)](\nu) \label{eq:fouriertu}
\end{align}
using the convolution theorem. 
Moreover, making the change of variable $x=-z$,
\begin{align} \mathfrak{F}[e^{-\alpha z}h(-z)](\nu) & =  \int_{-\infty}^{\infty} e^{ -{\bf i} \nu z} e^{ -\alpha z} h(-z) dz %
= \int_{-\infty}^{\infty} e^{ {\bf i} (\nu - {\bf i} \alpha) x} h(x) dx %
\nonumber \\ 
& =  \phi( \nu - {\bf i} \alpha) \label{eq:fourierdf}
\end{align}
where 
$   \phi(\nu) = \exp\left( - \frac{1}{2} \Delta_i \nu ^2 \right) $
is the characteristic function of the density $h$. 

The equality of equation (\ref{eq:fourierdf}) is well-defined since
$\norm{ \phi( \nu - {\bf i} \alpha) } < \infty$ for any $\alpha\in\mathbb{R}$. We introduce the dampening parameter $\alpha$ to ensure relative periodicity for the functions $u^{\alpha}_{i+1}$ as shown in the sequel. In practice, integrability is not necessary since a truncation is performed in the numerical implementation.
Combining equations (\ref{eq:fouriertu}) and (\ref{eq:fourierdf}) gives \begin{equation} \mathfrak{F}[ \tilde{ u}^{\alpha}_i ](\nu) = \mathfrak{F}[u^{\alpha}_{i+1}](\nu)  \phi( \nu - {\bf i} \alpha). \label{eq:ftua}\end{equation} 
\label{page:int}

Similarly, the Fourier transform of $\dot{u}_{i}^{\alpha}$ in equation~(\ref{eq:guf})
is given by
\begin{eqnarray} \mathfrak{F}[ \dot{ u}^{\alpha}_i ](\nu) & = & -\frac{1}{\Delta_i}\mathfrak{F}[u^{\alpha}_{i+1}](\nu) \mathfrak{F}[z e^{-\alpha z}h(-z)](\nu)  \nonumber \\
& = & -  \frac{{\bf i}}{\Delta_i} \mathfrak{F}[u^{\alpha}_{i+1}](\nu)  \Done{}{\nu} \mathfrak{F}[e^{-\alpha z}h(-z)](\nu) \nonumber \\
& = & -\frac{{\bf i}}{\Delta_i} \mathfrak{F}[u^{\alpha}_{i+1}](\nu)  \Done{}{\nu} \phi( \nu - {\bf i} \alpha) \nonumber \\
& = &  (\alpha + {\bf i } \nu) \mathfrak{F}[u^{\alpha}_{i+1}](\nu)  \phi( \nu - {\bf i} \alpha) \label{eq:fnua} \end{eqnarray}
using the differentiation properties of the Fourier transform.

From equations (\ref{eq:ftua}) and (\ref{eq:fnua}), we recover the
functions $\tilde{u}_{i}$ and $\dot{u}_{i}$ by taking the inverse
Fourier transform and adjusting for the dampening factor \begin{eqnarray} \tilde{u}_i (x) & = &  e^{\alpha x} \mathfrak{F}^{-1} \left[ \mathfrak{F}[u^{\alpha}_{i+1}](\nu)  \phi( \nu - {\bf i} \alpha) \right](x)  \label{eq:solu1}\\
\dot{ u}_i(x) & = & e^{\alpha x} \mathfrak{F}^{-1} \left[ (\alpha + {\bf i} \nu) \mathfrak{F}[u^{\alpha}_{i+1}](\nu) \phi( \nu - {\bf i} \alpha) \right](x). \label{eq:solu2}\end{eqnarray}
Equations (\ref{eq:uf}), (\ref{eq:solu1}), and (\ref{eq:solu2}), evaluated at $x=W_{t_{i}}$, define a convolution method for the approximate solution of the BSDE (\ref{eq:bsde}) based on the explicit Euler scheme II.

\subsection{Convolution on the explicit Euler scheme I}

An alternative characterization of the approximate solution of the BSDE~(\ref{eq:bsde}) is obtained if one considers the explicit
Euler scheme I of equation (\ref{eq:expdisc1}).  In this case, the approximate solution $(Y,Z)$
consists of functions $v_{i}$ and $\dot{v}_{i}$ at mesh time $t_{i}$
which take the form 
\begin{align} v_i(x) & =   \CEsp{}{ \tilde{v}_{i+1}(W_{t_{i+1}})  | W_{t_i} = x }{} \nonumber \\
& =  \int_{-\infty}^{\infty} \tilde{v}_{i+1}(y) h(y-x) dy \label{eq:3.18}
\end{align}
where 
\begin{align} 
\tilde{v}_{i+1}(x)  & =  v_{i+1}(x) + \Delta_i f(t_i, v_{i+1} (x), \dot{v}_i(x) ) \text{, } \label{eq:3.19} \\
\dot{v}_i(x) & =   \frac{1}{\Delta_i} \CEsp{}{ v_{i+1}(W_{t_{i+1}}) \Delta W_i | W_{t_i} = x }{} \nonumber \\
& =  \frac{1}{\Delta_i} \int_{-\infty}^{\infty}(y-x) v_{i+1}(y) h(y - x) dy, \label{eq:3.20}
\end{align}
for $i=0,1,...,n-1$ and $v_{n}(x)=g(x)$. 

Following the steps of the previous characterization equations (\ref{eq:3.18}) and (\ref{eq:3.20}) lead naturally to 
\begin{align} 
v_i (x) & =   e^{\alpha x} \mathfrak{F}^{-1} \left[ \mathfrak{F}[\tilde{v}^{\alpha}_{i+1}](\nu)  \phi( \nu - {\bf i} \alpha) \right](x)  \label{eq:solv1} \\
\dot{ v}_i(x) & =  e^{\alpha x} \mathfrak{F}^{-1} \left[ (\alpha + {\bf i} \nu) \mathfrak{F}[v^{\alpha}_{i+1}](\nu)  \phi( \nu - {\bf i} \alpha) \right](x). \label{eq:solv2} 
\end{align} 
where both $v_{i}^{\alpha}$ and $\tilde{v}_{i}^{\alpha}$ for $i=0,1,...,n-1$ along with the dampened terminal condition are
assumed to be integrable so that they admit Fourier transforms.  Equations (\ref{eq:3.19}), (\ref{eq:solv1}), and (\ref{eq:solv2}) define a convolution method for the approximate solution of the BSDE~(\ref{eq:bsde}) based on the explicit Euler scheme I.

\subsection{Numerical implementation} \label{sec:3.3}
From equations (\ref{eq:solu1}) and (\ref{eq:solu2}) or equations (\ref{eq:solv1}) and (\ref{eq:solv2})
one notices that computing the approximate solutions ($u_{i}$, $\dot{u}_{i}$) and
($v_{i}$, $\dot{v}_{i}$) at mesh time $t_{i}$ reduces to computing a function $\theta:\mathbb{R}\rightarrow\mathbb{R}$
depending on two functions $\psi:\mathbb{C}\rightarrow\mathbb{C}$
and $\eta:\mathbb{R}\rightarrow\mathbb{R}$ in the following manner
\begin{equation} \theta (x) = \frac{1}{2\pi} \int_{-\infty}^{\infty} e^{ {\bf i} \nu x } \widehat{\eta^{\alpha}}(\nu) \psi( \nu ) d\nu  \label{eq:Pint}\end{equation}if
we drop the dampening factor $e^{\alpha x}$.

This integral is numerically computed by discretizing the Fourier
space with a uniform grid of $N+1$ points $\{\nu_{i}\}_{i=0}^{N}$
on the interval $[-\frac{L}{2},\frac{L}{2}]$ of length $L$, where
$N$ is even, such that \begin{equation} \nu_i = \nu_0 + i \Delta \nu \end{equation}
with $\nu_{0}=-\frac{L}{2}$ and $\Delta\nu=\frac{L}{N}$. Hence,
for any $x\in\mathbb{R}$ 
\begin{eqnarray} \theta (x) & \approx & \frac{1}{2 \pi} \int_{ - \frac{L}{2} }^{ \frac{L}{2} } e^{ {\bf i} \nu x } \widehat{\eta^{\alpha}}(\nu) \psi( \nu ) d\nu \nonumber \\
& \approx & \frac{ \Delta \nu }{2 \pi } \sum_{i=0}^{N-1} e^{ {\bf i} \nu_i x } \widehat{ \eta^{\alpha} }( \nu_i) \psi(\nu_i) \label{eq:intctheta}
\end{eqnarray}where the integral is approximated using lower Riemann sums and 
\begin{equation} 
  \widehat{\eta^{\alpha}}(\nu_i) = \int_{-\infty}^{\infty} e^{- {\bf i} x \nu_i } \eta^{\alpha}(x) dx = \int_{-\infty}^{\infty} e^{- {\bf i} x \nu_i } e^{-\alpha x} \eta(x) dx . 
  \label{eq:feta}
\end{equation}This
last integral is also computed using an uniform grid of $N+1$ points 
$\{x_{j}\}_{j=0}^{N}$ on the restricted interval $[x_0,x_N]$ centred at $W_0=0$ such that \begin{equation} x_j = -\frac{l}{2} + j \Delta x \end{equation}
where $\Delta x = \frac{l}{N}$ is chosen so that the Nyquist relation
 $Ll = 2 \pi N$ is satisfied. 

The discretization of the integral in equation (\ref{eq:feta}) leads to an expression involving the discrete Fourier transform (DFT). The DFT is a numerical procedure that transforms a set of real or complex numbers $\{x_{j}\}_{j=0}^{N-1}$ into another set
$\{\hat{x}_{j}\}_{j=0}^{N-1}$ through the relation\begin{equation} \hat{x}_k := \mathfrak{D}[ x ]_k = \frac{1}{N} \sum_{j=0}^{N-1} e^{-{\bf i} j k \frac{2\pi}{N} } x_j \end{equation}for
$k=0,1,...,N-1$. The inverse DFT performs a reciprocal operation
by computing the set of numbers $\{{x}_{j}\}_{j=0}^{N-1}$ using
the numbers $\{\hat{x}_{j}\}_{j=0}^{N-1}$ as \begin{equation} x_k := \mathfrak{D}^{-1}[ \hat{x} ]_k = \sum_{j=0}^{N-1} e^{{\bf i} j k \frac{2\pi}{N} } \hat{x}_j \end{equation}for
$k=0,1,...,N-1$.

As we intend to use the DFT to compute (\ref{eq:intctheta}) and (\ref{eq:feta}) we assume that the following conditions are satisfied.  
\begin{assump} The generic dampened function $\eta^{\alpha}$ and its derivative $\Done{\eta^{\alpha}}{x}$ have the same values at the boundaries of the restricted domain $[x_0,x_N] $\begin{eqnarray}   \eta^{\alpha}(x_0) & =& \eta^{\alpha}(x_N), \label{eq:percond1}\\
\Done{ \eta^{\alpha} }{x}(x_0)  & = & \Done{ \eta^{\alpha} }{x}(x_N). \label{eq:percond2} \end{eqnarray}
\label{assump:boundcond}
\end{assump}  
We approximate the integral of equation (\ref{eq:feta}) by
first restricting the integration interval to $[x_{0},x_{N}]$ and
then applying a composite quadrature rule with weights $\{w_{i}\}_{i=0}^{N}$.
Consequently, \begin{eqnarray} \widehat{\eta^{\alpha}}(\nu_i) & \approx & \int_{x_0}^{x_N}  e^{- {\bf i} x \nu_i } \eta^{\alpha}(x) dx \label{eq:dgtrunc}  \\
& \approx & \Delta x \sum_{j = 0}^{N} w_j e^{ -{\bf i} x_j \nu_i} \eta^{\alpha}(x_j) \nonumber \\
& = & \Delta x \cdot e^{ -{\bf i} x_0 \nu_i }\sum_{j = 0}^{N} w_j e^{ -{\bf i} ji \frac{ 2 \pi}{N} } e^{ -{\bf i} j \nu_0 \Delta x } \eta^{\alpha}(x_j) \nonumber \\
& = & \frac{2\pi}{\Delta \nu} e^{ -{\bf i} x_0 \nu_i } \mathfrak{D} \left[ \{ (-1)^j \tilde{w}_j \eta^{\alpha}(x_j)\}_{j=0}^{N-1} \right]_i  \label{eq:ceta}
\end{eqnarray}
for $i=0,1,...,N-1$ using Assumption \ref{assump:boundcond} since $N$ is even and $e^{ -{\bf i} \nu_0 \Delta x} = -1$ with \begin{equation} \tilde{w}_j = w_j + \delta_{N-j,N} w_N \end{equation}
where $\delta_{i,j}$ stands for the Kronecker's delta. At this point many quadrature rules are available. For example, one may use the composite trapezoidal
rule with weights of the form \begin{equation} w_i = 1 - \frac{1}{2}( \delta_{0,i} + \delta_{N,i}) \text{ , } i = 0,1,...,N \end{equation}
leading to $\tilde{w}_{i}=1$. Higher order composite quadrature rules will improve accuracy in presence of a smooth driver $f$.

A similar approach can be found in \cite{lordal:2008} which enhanced the discrete Fourier transform with a composite trapezoidal
quadrature rule to compute this last integral. However, \cite{lordal:2008} omits Assumption \ref{assump:boundcond} making the error analysis 
quite tedious and leading to considerable numerical errors, especially around the boundaries of the restricted
domain. In addition, we use a fixed space grid whereas \cite{lordal:2008} shift the space grid through time steps. 
The major difference, however, is that the scheme in \cite{lordal:2008} solves the Snell envelope and does not seek a numerical solution for BSDEs.

\begin{rmk}
From equation (\ref{eq:dgtrunc}), the restriction of the real line to the interval $[x_0,x_N]$ also solves the problem of integrability of the generic dampened function $\eta^{\alpha}$ that we pointed out preceding equation (\ref{eq:ftua}). Indeed, the restriction is essentially equivalent to considering the truncated function $\eta^\alpha \mathbf{1}_{[x_0,x_N]}$. Hence, $\eta^\alpha$ needs to be integrable only on the restricted domain.
\end{rmk}

The values of the functions $\theta$ are computed at the grid points
$\{x_{k}\}_{k=0}^{N-1}$ by combining equations (\ref{eq:intctheta})
and (\ref{eq:ceta}) 
\begin{eqnarray} \theta (x_k) & \approx & \sum_{j=0}^{N-1} e^{ {\bf i} \nu_j x_k }  \psi(\nu_j)  e^{ -{\bf i} x_0 \nu_j } \mathfrak{D}\left[ \{ (-1)^i \tilde{w}_i \eta^{\alpha}(x_i)\}_{i=0}^{N-1} \right]_j \nonumber \\
& = & e^{ {\bf i} k\nu_0 \Delta x} \sum_{j=0}^{N-1} e^{ {\bf i} jk \frac{2\pi}{N} }  \psi(\nu_j) \mathfrak{D}\left[ \{ (-1)^i \tilde{w}_i \eta^{\alpha} (x_i)\}_{i=0}^{N-1} \right]_j   \nonumber\\
& = & (-1)^k \mathfrak{D}^{-1}\left[ \left\{ \psi(\nu_j) \mathfrak{D}\left[ \{ (-1)^i \tilde{w}_i \eta^{\alpha} (x_i)\}_{i=0}^{N-1} \right]_j     \right\}_{j=0}^{N-1} \right]_k . \label{eq:compformula}
\end{eqnarray}
Since we use the DFT, the underlying trigonometric (and hence periodic)
interpolation allows us to set \begin{equation} \theta(x_N) = \theta(x_0). \end{equation}

In applications we shall consider functions $\eta^{\alpha}$ that do not satisfy Assumption~\ref{assump:boundcond}. To address this problem
we slightly modify the function $\eta$ by adding a linear function to obtain a modified dampened function $\eta_{\beta,\kappa}^{\alpha}$
defined as \begin{equation} \eta^{ \alpha }_{\beta, \kappa} (x) = e^{-\alpha x}( \eta (x) + \beta x + \kappa ). \label{eq:fintr} \end{equation}
The following lemma gives the appropriate choice for the dampening parameter
$\alpha\in\mathbb{R}$, and the coefficients $\beta\in\mathbb{R}$
and $\kappa\in\mathbb{R}$.
\begin{lem}
\label{lem:coefopt}
Suppose the real function $\eta\in\mathcal{C}^{1}[a,b]$ is differentiable
with \begin{equation*}\Done{\eta}{x}(a) \neq \Done{\eta}{x}(b) \end{equation*} 
and let $\eta_{\beta,\kappa}^{\alpha}$ be the modified, dampened
function defined in equation (\ref{eq:fintr}). Then \begin{eqnarray} \alpha & = & \frac{1}{b-a} \log \left( \frac{ \Done{\eta}{x}(b) + \beta }{ \Done{\eta}{x}(a) + \beta} \right)  \label{eq:optalpha} \\  
\kappa & = & \frac{ e^{-\alpha b} ( \eta(b) + \beta b) - e^{-\alpha a}(\eta(a) + \beta a)}{ e^{-\alpha a} - e^{-\alpha b}} \label{eq:optkappa}
\end{eqnarray}solve the system of nonlinear equations \begin{eqnarray}   \eta_{\beta,\kappa}^{\alpha}(a) & = & \eta_{\beta,\kappa}^{\alpha}(b) \label{eq:sysfft1}\\
\Done{ \eta_{\beta,\kappa}^{\alpha} }{x}(a) & = & \Done{ \eta_{\beta,\kappa}^{\alpha} }{x}(b)  \label{eq:sysfft2} \end{eqnarray} for any $\beta\notin\{-\Done{\eta}{x}(a),-\Done{\eta}{x}(b)\}$. In addition, if \begin{equation} \beta > \max \left( | \Done{\eta}{x}(b) |, |\Done{\eta}{x}(a) | \right) \label{eq:optbeta} \end{equation}
then $\alpha\in\mathbb{R}$ and $\kappa\in\mathbb{R}$.\end{lem}
\begin{proof}
Equation (\ref{eq:sysfft1}) gives (\ref{eq:optkappa})
in a straightforward manner using basic algebra. Since $\eta$ is differentiable $\eta_{\beta,\kappa}^{\alpha}$
is also differentiable and 
\begin{equation*} 
\Done{ \eta_{\beta,\kappa}^{\alpha} }{x}(x) 
= - \alpha \eta_{\beta,\kappa}^{\alpha}(x) + e^{-\alpha x}\left( \Done{\eta}{x}(x)  + \beta \right)    
\end{equation*} 
and combining both equations (\ref{eq:sysfft1}) and (\ref{eq:sysfft2}) leads to (\ref{eq:optalpha}). 
Clearly, if the inequality (\ref{eq:optbeta}) holds then both $(\Done{\eta}{x}(b) + \beta)$
and $(\Done{\eta}{x}(a) + \beta)$ are strictly positive and $\alpha\in\mathbb{R}$.
\qquad\end{proof}

The transform of equation (\ref{eq:fintr}) may seem over parametrized. However, using only two parameters may lead to complex parameters or to an inconsistent system.

\begin{rmk}
When implementing the method, the values of the derivative $\Done{\eta}{x}$
at $x_{0}$ and $x_{N}$ can be approximated by forward or backward finite differences. 
\end{rmk}

\begin{rmk} \label{rem:slope}
A positive constant, $\epsilon >0$,  which represents the minimal slope
allowed in the linear transform $\beta x+\kappa$ is used to ensure equation~(\ref{eq:optbeta}) is enforced.
Set \begin{equation} \beta=\epsilon+\max\left(|\Done{\eta}{x}(x_N)|,|\Done{\eta}{x}(x_0)|\right). \end{equation}
\end{rmk}

Under the transformation of equation (\ref{eq:fintr}), the computation
of our approximate solution is not significantly more complex. Since the function $\theta$ is a (dampened) conditional expectation, properties of the conditional expectation allows the necessary adjustments as shown in the following theorem.
\begin{thm}
\label{thm:adjust} Let $\eta:[a,b]\rightarrow\mathbb{R}$ be an integrable
function and define $\eta_{\beta,\kappa}:[a,b]\rightarrow\mathbb{R}$
as 
\begin{equation*} 
\eta_{\beta,\kappa}(x) = \eta(x) + \beta x + \kappa 
\end{equation*}
such that $\eta_{\beta,\kappa}^{\alpha}$ is the modified, dampened
function of $\eta$ defined by equation (\ref{eq:fintr}). Then
the function $\theta : [a,b]\rightarrow\mathbb{R}$ of equation (\ref{eq:Pint})
admits the alternative representation 
\begin{equation} \theta(x) 
= \frac{1}{2\pi} \int_{-\infty}^{\infty} e^{ {\bf i} \nu x } \widehat{\eta^{\alpha}_{ \beta, \kappa}} (\nu) \psi( \nu ) d\nu - H(x,\alpha, \beta, \kappa) 
\end{equation} 
where 
\begin{equation} 
H(x,\alpha,\beta,\kappa)
= \begin{cases}  e^{-\alpha x} \beta &\text{  if }
\psi(\nu) = (\alpha + {\bf i} \nu) \phi(\nu - {\bf i} \alpha) \\  e^{-\alpha x} (\beta x + \kappa) &\text{  if }
\psi(\nu) = \phi(\nu - {\bf i} \alpha) . 
\end{cases} \label{eq:Hfunc}
\end{equation}
\end{thm} 
\begin{proof}
First, let $\psi(\nu) = (\alpha + {\bf i} \nu) \phi(\nu - {\bf i} \alpha)$.
By definition, we know that 
\begin{eqnarray*} \theta(x) & = & \frac{1}{2\pi} \int_{-\infty}^{\infty} e^{ {\bf i} \nu x } \widehat{\eta^{\alpha}}(\nu) \psi( \nu ) d\nu \nonumber \\
& = & \frac{ e^{-\alpha x} }{ \Delta_i } \CEsp{}{ \eta(W_{t_{i+1} }) \Delta W_i | W_{t_i} = x}{}  \\
& = & \frac{ e^{-\alpha x} }{ \Delta_i } \left(  \CEsp{}{ \left(\eta(W_{t_{i+1} }) + \beta W_{t_{i+1}} + \kappa \right)\Delta W_i | W_{t_i} = x}{}  - \beta \Delta_i \right) \\
& = & \frac{ e^{-\alpha x} }{ \Delta_i } \CEsp{}{ \eta_{\beta,\kappa}(W_{t_{i+1} }) \Delta W_i | W_{t_i} = x}{} - e^{-\alpha x} \beta \\
& = & \frac{1}{2\pi} \int_{-\infty}^{\infty} e^{ {\bf i} \nu x } \widehat{\eta^{\alpha}_{ \beta, \kappa}} (\nu) \psi( \nu ) d\nu - e^{-\alpha x} \beta. 
\end{eqnarray*}

Similarly, if $\psi(\nu) = \phi(\nu - {\bf i} \alpha)$, we have 
\begin{eqnarray*} \theta(x) & = & \frac{1}{2\pi} \int_{-\infty}^{\infty} e^{ {\bf i} \nu x } \widehat{\eta^{\alpha}}(\nu) \psi( \nu ) d\nu  \\
& = & { e^{-\alpha x} } \CEsp{}{ \eta(W_{t_{i+1} })  | W_{t_i} = x}{}  \\
& = & { e^{-\alpha x} } \CEsp{}{ \eta_{\beta,\kappa}(W_{t_{i+1} }) | W_{t_i} = x}{} - e^{-\alpha x} (\beta x + \kappa) \\
& = & \frac{1}{2\pi} \int_{-\infty}^{\infty} e^{ {\bf i} \nu x } \widehat{\eta^{\alpha}_{ \beta, \kappa}} (\nu) \psi( \nu ) d\nu - e^{-\alpha x} (\beta x + \kappa). \end{eqnarray*}
\end{proof}

Theorem \ref{thm:adjust} shows that the computational formula of equation (\ref{eq:compformula}) can be applied to the dampened transform function $\eta^{\alpha}_{\beta, \kappa}$ which satisfies Assumption \ref{assump:boundcond} after an appropriate choice of the coefficients $\alpha$, $\beta$ and $\kappa$ using Lemma~\ref{lem:coefopt}. One recovers the function $\theta$ by subtracting the function $H$. 

The implementation of the convolution method gives the numerical approximations $\{ u_{ik} \}_{k=0}^N$ and $\{ \dot{u}_{ik} \}_{k=0}^N$ to the approximate solutions $u_i$ and $\dot{u}_i$ at space  nodes $\{ x_k \}_{k=0}^N$ for any  time node $t_i$, $i = 0,1,2,...,n-1$. An approximate solution to the BSDE consists of a (linear) interpolation of a Brownian path through the node values $\{ u_{ik} \}_{k=0}^N$ and $\{ \dot{u}_{ik} \}_{k=0}^N$ for $i=0,1,2,...,n-1$.  The following algorithm summarizes the steps necessary to implement the convolution method.

\begin{algo} Convolution Method
\begin{enumerate}
\item Discretize the restricted real space $[-\frac{l}{2},\frac{l}{2}]$ and the restricted Fourier space $[-\frac{L}{2},\frac{L}{2}]$ with $N$ space steps so to have the real space nodes $\{ x_{k}\}_{k=0}^{N}$ and the Fourier space nodes $\{ \nu_{k}\}_{k=0}^{N}$
\item Set $u_n(x_{k}) = g(x_{k})$
\item For any $i$ from $n-1,\ldots,0$
\begin{enumerate}
\item Compute $\alpha$, $\beta$ and $\kappa$ using Lemma \ref{lem:coefopt}, such that
\begin{equation} \eta^{\alpha}_{}  = \left( u_{i+1} \right)^{\alpha}_{\beta,\kappa} \end{equation} and $\eta^{\alpha}_{}$ satisfies the boundary conditions.
\item Compute $\theta(x_{k})$ through equation (\ref{eq:compformula}) for $k=0,1,...,N$ with \begin{equation} \psi(\nu) = \phi(\nu - {\bf i} \alpha) \end{equation} and retrieve the values $\tilde{u}_{ik}$ as \begin{equation} \tilde{u}_{ik} = e^{\alpha x_{k}}\theta(x_{k}) -  \left( \beta x_{k} + \kappa \right)  \end{equation} through Theorem \ref{thm:adjust}.
\item Compute $\theta(x_{ik})$ through equation (\ref{eq:compformula}) for $k=0,1,...,N$ with \begin{equation} \psi(\nu) = (\alpha + {\bf i}\nu)\phi(\nu - {\bf i} \alpha) \end{equation} and retrieve the values $\dot{u}_{ik}$ as \begin{equation} \dot{u}_{ik} = e^{\alpha x_{k}}\theta(x_{k}) -  \beta  \end{equation}
through Theorem \ref{thm:adjust}.
\item Compute the values $u_{ik}$ as 
\begin{equation} u_{ik} = \tilde{u}_{ik} + \Delta_i f(t_i,\tilde{u}_{ik}, \dot{u}_{ik}) \end{equation} for $k=0,1,...,N_iN$ through equation (\ref{eq:uf}) when using the explicit Euler scheme II.%
\end{enumerate}
\end{enumerate}
\end{algo}

We next consider the problem of computational error. 

\section{Error Analysis} \label{sec:error}
The convolution method induces two main types of error.  In addition to the time discretization error $E_{\pi}$ discussed in Section \ref{sec:timedisc}, there is a space discretization error and we focus on this last error term. We limit our analysis to the explicit Euler scheme II since equivalent results are obtained for the explicit Euler scheme I using the same techniques.

Throughout this section $\{ {\bf{u}}_{ik} \}_{k=0}^N$, $\{ \tilde{\bf{u}}_{ik} \}_{k=0}^N$ and $\{ \dot{\bf{u}}_{ik} \}_{k=0}^N$ denote the numerical solution obtained from the convolution method at time mesh $t_i$, $i=0,1,...,n-1$ knowing the approximate solutions $u_{i+1}(x_k)$ and $\dot{u}_{i+1}(x_k)$ at time mesh $t_{i+1}$.
The convolution method induces a space discretization error when approximating the values of $u_i(x_k)$ and $\dot{u}_i(x_k)$ by $u_{ik}$ and $\dot{u}_{ik}$ respectively. We will particularly describe the local behaviour of this error term.  We define it as 
\begin{equation} E_{ik} :=  \norm{u_i(x_k) - {\bf{u}}_{ik} } + \norm{\dot{u}_i(x_k) - \dot{\bf{u}}_{ik} } \label{eq:localDE}. \end{equation}
The following lemma describes the DFT accuracy in approximating the Fourier coefficients and proves useful in the derivation of a space discretization error bound. We skip the proof since the results are well known (see \cite[ Theorem 3.4, p.\ 140]{plato:2003} and  \cite[Theorem 4.4, p.\ 85]{vretblad:2003}).
\begin{lem} \label{lem:fourier} 
Suppose the integrable function 
$u: \left[- \frac{l}{2},\frac{l}{2}\right]\rightarrow \mathbb{R}$ with\\ $u(-\frac{l}{2})= u(\frac{l}{2})$ 
admits the Fourier series expansion 
\begin{equation} 
u(x) = \sum_{k=-\infty}^{\infty} c_k e^{ {\bf i} k \frac{2\pi}{l} x } \text{, } x \in \left[- \frac{l}{2},\frac{l}{2}\right] 
\end{equation}  
and  $\{x_k\}_{k=0}^{N-1}$ are the nodes of the equidistant grid of $\left[- \frac{l}{2},\frac{l}{2}\right]$ 
such that \\$x_k = -\frac{l}{2} + k \Delta $ where  $N \in \mathbb{N}$ is even and $\Delta  = N^{-1}l$. 

If $u \in \mathcal{C}^2$ then 
\begin{equation}  
c_{k - \frac{N}{2}} = (-1)^{ k - \frac{N}{2} }\mathcal{D}\left[ \{ (-1)^j u(x_j)  \}_{j=0}^{N-1} \right]_k + \mathcal{O}( \Delta^2 ) 
\end{equation} 
for $k=0,1,...,N-1$ and 
\begin{equation} 
\norm{c_k} < C k^{-2} \text{ ,} 
\end{equation}  
for $k \in \mathbb{Z} \backslash \{0\}$ and some constant $C>0$ depending on $\Dtone{u}{x}$. 
Consequently, 
\begin{equation} 
u(x) = \sum_{k= -\frac{N}{2} }^{\frac{N}{2} - 1} c_ke^{ {\bf i} k \frac{2\pi}{l} x } + \mathcal{O}(\Delta) \text{, } \forall x \in \left[- \frac{l}{2},\frac{l}{2}\right].  
\end{equation}
\end{lem}

The next theorem gives an error bound for the space discretization error under smoothness conditions on the BSDE coefficients $f$ and $g$.
\begin{thm} \label{thm:LDE}
Suppose $f \in \mathcal{C}^{1,2,2}$ and $g \in \mathcal{C}^2$. Then for any $i=0,1,...,n-1$ and $k=0,1,...,N$, the convolution method applied on the truncated interval $\left[ -\frac{l}{2},\frac{l}{2}\right]$ yields a (local) discretization error of the form
\begin{equation} E_{ik} = \chi(x_k) +  \mathcal{O}\left( \Delta x \right) + \mathcal{O}\left( e^{-K \Delta^{-1}_i l^2} \right) \end{equation}  where the extrapolation error $\chi$ satisfies \begin{equation} \norm{\chi(x_k)} \leq C \left(  \left( \int_{ \frac{l}{2} - \norm{x_k} }^{\frac{l}{2}}h(y)dy \right)^{\frac{1}{2}} +  \Delta^{\frac{1}{2}} \right) \label{eq:extraperr} \end{equation} for some positive constants $C,K>0$ depending on the driver $f$, the function $g$, and the terminal time $T$ when using the trapezoidal quadrature rule.
\end{thm}
\begin{proof} Suppose the solution $u_{i+1}$ at time $t_{i+1}$ is known. Since $f \in \mathcal{C}^{1,2,2}$ and $g \in \mathcal{C}^2$, it is easily shown that $u_{i+1} \in \mathcal{C}^2$. Also, we know from \cite{zhang:2004} and \cite{bouchardtouzi:2004} that $ Y_{t_{i+1}}^{\pi} = u_{i+1}(W_{t_{i+1}})$ is square integrable so that $u_{i+1}$ is square integrable (with respect to the Gaussian density). 

In light of Theorem~\ref{thm:adjust}, we can limit ourselves to the case where 
\begin{equation*} 
u_{i+1} \left( -\frac{l}{2} \right) 
= u_{i+1} \left( \frac{l}{2} \right) \text{ and } \Done{u_{i+1}}{x}\left( -\frac{l}{2} \right) = \Done{u_{i+1}}{x}\left( \frac{l}{2} \right) 
\end{equation*} so that $\alpha =\beta=\kappa=0$. 
Let $\{c_k\}_{k=-\infty}^{\infty}$ be the Fourier coefficients of $u_{i+1}$ on $\left[ -\frac{l}{2},\frac{l}{2}\right]$. We have that
\begin{eqnarray*}
\tilde{u}_{i}(x_k) & = & \int_{ \norm{y-x_k} \leq \frac{l}{2} } u_{i+1}(y) h(y-x_k)dy + \int_{ \norm{y-x_k} > \frac{l}{2} } u_{i+1}(y) h(y-x_k)dy \\
& = & \int_{ \norm{y}  \leq \frac{l}{2} } u_{i+1}(x_k + y) h(y)dy + \int_{ \norm{y} > \frac{l}{2} } u_{i+1}( x_k+y) h(y)dy 
\end{eqnarray*} 
where \begin{eqnarray*} \int_{ \norm{y} > \frac{l}{2} } u_{i+1}( x_k+y) h(y)dy & = & \CEsp{}{u_{i+1}(x_k + \Delta W_{n-1}){\bf 1}_{\mathbb{R} \backslash [-\frac{l}{2},\frac{l}{2}] } (\Delta W_{n-1} )}{} \\
& = & \mathcal{O}\left(e^{-K {\Delta_i}^{-1} l^2}\right) \end{eqnarray*} for some constant $K>0$ by successively applying Cauchy-Schwartz and Chernoff inequalities since the solution $u_{i+1}$ is square integrable. Hence 
\begin{eqnarray} \tilde{u}_{i}(x_k) & = & \int_{ \norm{y}  \leq \frac{l}{2} } T_{\infty}(x_k+y) h(y)dy  + \mathcal{O}\left(e^{-K {\Delta_i}^{-1} l^2}\right) \nonumber \\
& & +  \int_{ \norm{y}  \leq \frac{l}{2} } \left(u_{i+1}(x_k + y) - T_{\infty}(x_k+y) \right) h(y)dy   \label{eq:errorinter} \end{eqnarray} where $T_{\infty}(x) = \sum_{k= -\infty }^{\infty} c_ke^{ {\bf i} k \frac{2\pi}{l} x }$ for $x \in \mathbb{R}$. So that, on one hand, we have 
\begin{eqnarray} 
\lefteqn{ \int_{ \norm{y}  \leq \frac{l}{2} } T_{\infty}(x_k+y) h(y)dy }\nonumber  \\ 
& = & \sum_{j = -\frac{N}{2}}^{\frac{N}{2}-1} c_{j} e^{{\bf i} j \frac{2\pi}{l} x_k} \phi\left(j \frac{2\pi}{l}\right) - \int_{ \norm{y} >\frac{l}{2} } T_{\infty}(x_k+y) h(y)dy +   \mathcal{O}(\Delta x) \nonumber \\
& & \text{(by Lemma \ref{lem:fourier}),} \nonumber \\  
& = & \sum_{j = -\frac{N}{2}}^{\frac{N}{2} - 1} \phi\left(j \frac{2\pi}{l}\right) c_{j} e^{{\bf i} j \frac{2\pi}{l} x_k}  +   \mathcal{O}(\Delta x) + \mathcal{O}\left(e^{-K\norm{\Delta_i}^{-1}l^2}\right) \nonumber \\
&   & \text{ (by boundedness of $T_{\infty}$ and Chernoff inequality), } \nonumber \\
& = & (-1)^k \sum_{j=0}^{N-1} \phi( \nu_j ) \mathcal{D}\left[ \{ (-1)^s u_{i+1} (x_s)\}_{s=0}^{N-1} \right]_j  e^{{\bf i} \frac{2\pi}{N} jk}  \nonumber \\
&  & + \text{ } \mathcal{O}(\Delta x) + \mathcal{O}\left(e^{-K {\Delta_i}^{-1}l^2}\right) \text{ (by Lemma  \ref{lem:fourier}),} \nonumber \\
& = & \tilde{\bf{u}}_{ik} + \mathcal{O}(\Delta x) + \mathcal{O}\left(e^{-K {\Delta_i}^{-1}l^2} \right)  \label{eq:prfloc1} .
\end{eqnarray} 
Assuming $x_k \geq 0$, without loss of generality, define $\chi_0$ as 
\begin{eqnarray*} 
\chi_0(x_k) & = & \int_{ \norm{y}  \leq \frac{l}{2} } \left(u_{i+1}(x_k + y) - T_{\infty}(x_k+y) \right) h(y)dy \\
& = & \int_{ \frac{l}{2} - x_k }^{\frac{l}{2}}  \left(u_{i+1}(x_k + y) - u_{i+1}(x_k + y - l) \right) h(y)dy 
\end{eqnarray*} 
since $T_{\infty}$ is periodic and $T_{\infty}(x)=u_{i+1}(x)$ on the interval $\left[-\frac{l}{2},-\frac{l}{2}\right]$. 
Equation (\ref{eq:errorinter}) becomes 
\begin{equation} 
  \tilde{u}_{i}(x_k) = \tilde{\bf{u}}_{ik} + \chi_0(x_k) + \mathcal{O}(\Delta x) 
  + \mathcal{O}\left(e^{-K {\Delta_i}^{-1}l^2}\right) 
\end{equation} 
and we note, by the continuity of $u_{i+1}$, that
\begin{equation} 
  \norm{\chi_0(x_k)}  \leq C_0 \int_{ \frac{l}{2} - \norm{x_k} }^{\frac{l}{2}}h(y)dy 
  \label{eq:prfloc2} 
\end{equation} 
for some positive constant $C_0>0$. 

Similarly
 \begin{eqnarray*}
\dot{u}_{i}(x_k) %
& = & \frac{1}{\Delta_i} \int_{ \norm{y}  \leq \frac{l}{2} } u_{i+1}(x_k + y) yh(y)dy +  \frac{1}{\Delta_i} \int_{ \norm{y} > \frac{l}{2} } u_{i+1}( x_k+y) yh(y)dy 
\end{eqnarray*} 
where \begin{eqnarray*}  \lefteqn{\frac{1}{\Delta_i} \int_{ \norm{y} > \frac{l}{2} } u_{i+1}( x_k+y)y h(y)dy } \\ & = &  \frac{1}{\Delta_i}\CEsp{}{u_{i+1}(x_k + \Delta W_{n-1})\Delta W_{n-1} {\bf 1}_{\mathbb{R} \backslash [-\frac{l}{2},\frac{l}{2}] } (\Delta W_{n-1} )}{} \\
& \leq &  \frac{K}{\Delta_i} \CEsp{}{ (\Delta W_{n-1})^2 {\bf 1}_{\mathbb{R} \backslash [-\frac{l}{2},\frac{l}{2}] } (\Delta W_{n-1} )}{}^{\frac{1}{2}} \\
&   & \text{ (by Cauchy-Schwartz inequality), } \\
& = & \mathcal{O}\left({\Delta_i}^{-\frac{1}{2}}e^{-K {\Delta_i}^{-1} l^2}\right) \\
& & \text{(by successively applying Cauchy-Schwartz and Chernoff inequalities),} \\
& = & \mathcal{O}\left(e^{-\frac{1}{2}K {\Delta_i}^{-1} l^2}\right). \end{eqnarray*}  Hence 
\begin{eqnarray} \dot{u}_{i}(x_k) & = & \frac{1}{\Delta_i} \int_{ \norm{y}  \leq \frac{l}{2} } T_{\infty}(x_k+y) yh(y)dy  + \mathcal{O}\left(e^{-K {\Delta_i}^{-1} l^2}\right) \nonumber \\
& & +  \frac{1}{\Delta_i} \int_{ \norm{y}  \leq \frac{l}{2} } \left(u_{i+1}(x_k + y) - T_{\infty}(x_k+y) \right) yh(y)dy.  \label{eq:errorinter2} \end{eqnarray} 
Letting $c^{\prime}_j = {\bf i}j \frac{2\pi}{l} c_j$, we have 
\begin{eqnarray} \lefteqn{ \frac{1}{\Delta_i} \int_{ \norm{y}  \leq \frac{l}{2} } T_{\infty}(x_k+y)y h(y)dy }\nonumber  \\ 
& = & \sum_{j = -\frac{N}{2}}^{\frac{N}{2}-1} {c^{\prime}_{j}} e^{{\bf i} j \frac{2\pi}{l} x_k} \phi\left(j \frac{2\pi}{l}\right) - \frac{1}{\Delta_i}\int_{ \norm{y} >\frac{l}{2} } T_{\infty}(x_k+y)y h(y)dy +   \mathcal{O}(\Delta x) \nonumber \\
& & \text{(by Lemma \ref{lem:fourier} since $x\phi(x)$ is bounded),} \nonumber \\  
& = & \sum_{j = -\frac{N}{2}}^{\frac{N}{2} - 1} \phi\left(j \frac{2\pi}{l}\right) {c^{\prime}_{j}} e^{{\bf i} j \frac{2\pi}{l} x_k}  +   \mathcal{O}(\Delta x) + \mathcal{O}\left(e^{-K{\Delta_i}^{-1}l^2}\right) \nonumber \\
&   & \text{ (by boundedness of $T_{\infty}$ and Chernoff inequality), } \nonumber \\
& = & (-1)^k \sum_{j=0}^{N-1} {\bf i} \nu_j \phi( \nu_j ) \mathcal{D}\left[ \{ (-1)^s u_{i+1} (x_s)\}_{s=0}^{N-1} \right]_j  e^{{\bf i} \frac{2\pi}{N} jk}  \nonumber \\
&  & + \text{ } \mathcal{O}(\Delta x) + \mathcal{O}\left(e^{-K {\Delta_i}^{-1}l^2}\right) \text{ (by Lemma  \ref{lem:fourier}),} \nonumber \\
& = & \dot{\bf{u}}_{ik} + \mathcal{O}(\Delta x) + \mathcal{O}\left(e^{-K {\Delta_i}^{-1}l^2} \right)  \label{eq:prfloc12} .
\end{eqnarray} 
By equations (\ref{eq:errorinter2}) and (\ref{eq:prfloc12})
\begin{equation} \dot{u}_{i}(x_k) = \dot{\bf{u}}_{ik} + \chi_1(x_k) + \mathcal{O}\left(\Delta x \right) + \mathcal{O}\left( e^{-K {\Delta_i}^{-1}l^2}\right) \label{eq:prfloc3} 
\end{equation} 
where $K>0$ and, letting $\upsilon(y) = u_{i+1}(x_k + y) - T_{\infty}(x_k + y)$,
\begin{eqnarray*} 
{\chi_1(x_k)}  &=&   \Delta^{-1}_i  \int_{ \norm{y}  \leq \frac{l}{2} } y \upsilon(y) h(y)dy \\
&=&   \Delta^{-1}_i  \int_{ \norm{y}  \leq \frac{l}{2} } y^2 \frac{\upsilon(y) - \upsilon(0)}{y} h(y)dy \\
&=&  \Delta^{-1}_i  \int_{ \norm{y}  \leq \frac{l}{2} } y^2 \left( \Done{\upsilon}{x}(y) + \Dtone{\upsilon}{x}(\xi) y \right) h(y)dy   
\end{eqnarray*}
for some $\xi \in \left[ -\frac{l}{2}, \frac{l}{2} \right]$. Since $\Done{T_{\infty}}{x}$ is the Fourier expansion of $\Done{u_{i+1}}{x}$ and $u,v \in \mathcal{C}^2$ , we have
\begin{eqnarray} 
\norm{\chi_1(x_k)}  & \leq &  C_1 \Delta_i^{-1} \left(   \int_{ \frac{l}{2} - \norm{x_k} }^{\frac{l}{2}}y^2h(y)dy + \int_{\norm{y}  \leq \frac{l}{2}} \norm{y}^3 h(y)dy    \right)\nonumber \\
 & \leq & C_1 \left(  \left( \int_{ \frac{l}{2} - \norm{x_k} }^{\frac{l}{2}}h(y)dy \right)^{\frac{1}{2}} +  \Delta^{\frac{1}{2}} \right)\label{eq:prfloc4} 
\end{eqnarray}
by the Cauchy-Schwartz inequality, for some constant $C_1>0$.

The Lipschitz property of the driver $f$ completes the proof from the relations in equations (\ref{eq:prfloc1}), (\ref{eq:prfloc2}), (\ref{eq:prfloc3}) and (\ref{eq:prfloc4}).
\end{proof}

Theorem \ref{thm:LDE} decomposes the spatial discretization error in three parts: the truncation error, the discretization error and the extrapolation error. Most PDE based and spatial discretization based methods for BSDEs fail in giving a bound for the error due to truncation. The error analysis shows that the truncation error $\mathcal{O}(e^{-Kl^2})$ has a spectral convergence of index 2 when applying the convolution method. Also, the discretization error $\mathcal{O}\left( \Delta x \right)$, of first order, is similar to PDE based methods such as \cite{douglas:1996} or  \cite{milstein:2007}.

The extrapolation error $\chi$ is specific to the convolution method implemented using the DFT.  Equation (\ref{eq:extraperr}) shows that errors appear and may accumulate around the boundaries of the truncated domain. Nonetheless, the extrapolation error is mainly time related through the density $h$ and can be confined at the boundaries for fine time discretizations as shown in the following corollary. 

\begin{cor}
\label{cor:linext}
Under the conditions of Theorem \ref{thm:LDE}, \begin{equation} \lim_{ \norm{\pi} \rightarrow 0 } \chi(x_k) = 0 \end{equation}
for any $x_k \in \left[-\frac{l}{2}, \frac{l}{2} \right]$.
\end{cor}
\begin{proof}
If $x_k = 0$ then equation (\ref{eq:extraperr}) gives $ \norm{ \chi(0) } \leq C \norm{\pi} ^{\frac{1}{2}}$ and the result holds.
If $x_k \neq 0$ and $x_k \in \left[-\frac{l}{2}, \frac{l}{2} \right]$, then
\begin{eqnarray*}  
\lim_{ \norm{\pi} \rightarrow 0 } \left( \int_{ \frac{l}{2} - \norm{x_k} }^{\frac{l}{2}}h(y)dy \right)^{\frac{1}{2}} & = & \left( \lim_{ \norm{\pi} \rightarrow 0 }\int_{ \frac{l}{2} - \norm{x_k} }^{\frac{l}{2}}h(y)dy \right)^{\frac{1}{2}} \\
& = & \left( \lim_{ \Delta_i \rightarrow 0 }\int_{ \frac{l}{2} - \norm{x_k} }^{\frac{l}{2}}h(y)dy \right)^{\frac{1}{2}} \\
& = & \left( \int_{ \frac{l}{2} - \norm{x_k} }^{\frac{l}{2}} \delta(y) dy \right)^{\frac{1}{2}} \\
& & \text{ (where $\delta$ is the Dirac delta function), }\\
& = & 0.
\end{eqnarray*}
Equation (\ref{eq:extraperr}) then leads to the result.
\end{proof}

We define the global convergence error for each space node as
\begin{equation}  
E_{l,\Delta x} = \sup_{i, k} e_{i,k} + \sup_{i, k} \dot{e}_{i,k} \label{eq:globalDR}
  \end{equation}
  where \begin{equation} e_{ik} = \norm{u_{n-i}(x_k) - {u}_{n-i,k}}  \label{eq:globalDR1} \end{equation} and \begin{equation} \dot{e}_{ik} = \norm{\dot{u}_{n-i}(x_k) - {\dot{u}}_{n-i,k}} \label{eq:globalDR2} \end{equation}
  for $i=1,...,n$ with $e_{0,k}=\dot{e}_{0,k}=0$. The next theorem
describes the stability and convergence properties of the convolution
method.
\begin{thm}
\label{thm:GDR}Suppose the conditions of Theorem \ref{thm:LDE}
are satisfied and let $h_{\Delta}$ be the Gaussian density with zero mean and variance $\Delta$. If the discretization is such that \begin{equation} \sup_{ i }  \max\left(\frac{\Delta x}{\sqrt{2\pi \Delta_i}} , \frac{\Delta x}{ \pi \Delta_i} \right)  \leq 1    \label{eq:ccond} \end{equation}then
the convolution method is stable and the global discretization error
$E_{l,\Delta x}$ satisfies 
\begin{equation}
 E_{l,\Delta x} =\chi_{\norm{\pi} } + \mathcal{O}(\Delta x ) + \mathcal{O}\left(e^{-K \norm{\pi}^{-1} l^{2}}\right) \label{eq:GDR} 
 \end{equation}
where \begin{equation} {\chi_{\norm{\pi}}} \leq C \left(  \left( \int_{ 0 }^{\frac{l}{2}}h_{\norm{\pi}}(y)dy \right)^{\frac{1}{2}} +  \norm{\pi}^{\frac{1}{2}} \right) 
\end{equation}  for some constants $C,K >0$ depending on the driver $f$, the terminal function $g$ and  the terminal time $T$. Finally $\lim_{ \norm{\pi} \rightarrow 0 } \chi_{\norm{\pi} } = 0$.
\end{thm}

\begin{proof}
First note that from the definitions of equations (\ref{eq:localDE}) and (\ref{eq:globalDR1})
\begin{eqnarray}
e_{ik} & \leq & { E}_{n-i,k} + \norm{ \mathbf{{u}}_{n-i,k} - {u}_{n-i,k}}  \nonumber\\
& \leq & E_{n-i,k} + (1+\Delta_iK) \norm{ \mathbf{\tilde{u}}_{n-i,k} - \tilde{u}_{n-i,k}} 
+ \Delta_i K \norm{ \mathbf{\dot{u}}_{n-i,k} - {\dot{u}}_{n-i,k}} \label{eq:prf1}
\end{eqnarray}where $K>0$ is the Lipschitz constant of the driver $f$. Also, we
have that \begin{equation} \dot{e}_{ik} \leq {E}_{n-i,k} + \norm{ \mathbf{\dot{u}}_{n-i,k} - {\dot{u}}_{n-i,k}}. \label{eq:prf12} \end{equation}from
equations (\ref{eq:localDE}) and (\ref{eq:globalDR2}).

Furthermore, the construction of the convolution method gives \begin{eqnarray} 
\norm{ \mathbf{\tilde{u}}_{i,k} - \tilde{u}_{i,k}} & \leq & \norm{ \mathfrak{D}^{-1}\left[ \left\{ \phi(\nu_{j}) \mathfrak{D}[ \left \{(-1)^s \tilde{w}_s ( u_{i+1}(x_s) - u_{i+1,s}) \right \}_{s=0}^{N-1}  ]_j     \right\}_{j=0}^{N-1} \right]_{k} }  + { E}_{i,k} \nonumber\\
&  & \text{ (by Theorem \ref{thm:LDE} since the transform functions are given), }\nonumber\\
& \leq & \frac{1}{N} \left( \sum_{j=0}^{N-1} \phi(\nu_{j}) \right) \sup_k \norm{ u_{i+1}(x_{i,k}) - u_{i+1,k} }+ { E}_{i,k} \nonumber\\
&  & \text{ (using the matrix-vector representation of DFTs), }\nonumber\\
& \leq & \frac{1}{N} \left( \sum_{j=0}^{N-1} \phi(\nu_{j}) \right) \sup_k e_{n-i-1,k} + { E}_{i,k} \nonumber\\ 
& \leq & \frac{ (\Delta \nu)^{-1} }{N} \left( \int_{\mathbb{R}} \phi(x)dx \right) \sup_k e_{n-i-1,k} + { E}_{i,k} \nonumber\\
& = & \frac{ \Delta x}{(2\pi \Delta_i)^{\frac{1}{2}}} \sup_k e_{n-i-1,k} + { E}_{i,k}. \label{eq:prf2}
\end{eqnarray}Similarly, \begin{eqnarray} 
\norm{ \mathbf{\dot{u}}_{i,k} - \dot{u}_{i,k}} & \leq &  \norm{ \mathfrak{D}^{-1}\left[ \left\{ \nu_j \phi(\nu_{j}) \mathfrak{D}[ \left \{(-1)^s \tilde{w}_s ( u_{i+1}(x_s) - u_{i+1,s}) \right \}_{s=0}^{N-1}  ]_j     \right\}_{j=0}^{N-1} \right]_{k} }  + { E}_{i,k} \nonumber\\
&  & \text{ (by Theorem \ref{thm:LDE} since the transform functions are given), }\nonumber\\
& \leq & \frac{1}{N} \left( \sum_{j=0}^{N-1} \norm{\nu_{j}}\phi(\nu_{j}) \right) \sup_k e_{n-i-1,k} + { E}_{i,k} \nonumber\\
&  & \text{ (using the matrix representation of DFTs), }\nonumber\\
& \leq & \frac{(\Delta \nu)^{-1} }{N} \left( \int_{\mathbb{R}} \norm{x}\phi(x)dx \right) \sup_k e_{n-i-1,k} + { E}_{i,k} \nonumber\\
& = & \frac{ \Delta x}{\pi \Delta_i} \sup_k e_{n-i-1,k} + { E}_{i,k}. \label{eq:prf3}
\end{eqnarray}Then, combining the inequalities of equations (\ref{eq:prf1}), (\ref{eq:prf2}) and (\ref{eq:prf3})
leads to \begin{eqnarray} 
e_{i,k} & \leq & C_0 { E}_{i,k} + \left({1+ 2\Delta_iK} \right) \max\left(\frac{\Delta x}{\sqrt{2\pi \Delta_i}} , \frac{\Delta x}{ \pi \Delta_i} \right)  \sup_k e_{i-1,k} \nonumber\\
& \leq & C_0 \sup_{i,k}{ E}_{i,k} + \left({1+ 2\Delta_i K} \right)\max\left(\frac{\Delta x}{\sqrt{2\pi \Delta_i}} , \frac{\Delta x}{ \pi \Delta_i} \right) \sup_k e_{i-1,k} \nonumber
\end{eqnarray}where $C_0>0$ and $K>0$ is the Lipschitz constant of the driver
$f$. So that \begin{eqnarray} \sup_k e_{i,k} & \leq &  C_0 \sup_{i,k}{ E}_{i,k} + \left({1+ 2\Delta_iK} \right) \max\left(\frac{\Delta x}{\sqrt{2\pi \Delta_i}} , \frac{\Delta x}{ \pi \Delta_i} \right)  \sup_k e_{i-1,k} \nonumber\\
& \leq &  C_0 \sup_{i,k}{ E}_{i,k} + ({1+ 2\Delta_i K}) \zeta \sup_k e_{i-1,k}
\label{eq:convprf} \end{eqnarray}for some positive number $\zeta$ satisfying \begin{equation*} \sup_i \max\left(\frac{\Delta x}{\sqrt{2\pi \Delta_i}} , \frac{\Delta x}{ \pi \Delta_i} \right) \leq \zeta \leq 1. \end{equation*}

From the inequality of equation (\ref{eq:convprf}), Gronwall's Lemma
yields \begin{equation} \sup_k e_{i,k} \leq C_0 e^{2T K}   \sup_{i,k} { E}_{i,k} \label{eq:prf4} \end{equation}
for $i=0,1,...,n$ knowing that $e_{0,k}=0$. Hence, the convolution
method is stable for the approximate solution $u_{i}$ since its error
at any time step is absolutely bounded.

The inequalities of equations (\ref{eq:prf12}), (\ref{eq:prf3}) and (\ref{eq:prf4})
lead to \begin{eqnarray} \sup_k \dot{e}_{i,k} & \leq & \left( C_1 +  \frac{\Delta x}{\pi \Delta_i}  {C_0 e^{2T K}} \right) \sup_{i,k}E_{i,k} \nonumber\\
& \leq & \left( C_1 +  {C_0 e^{2T K}} \right) \sup_{i,k}E_{i,k} \label{eq:prf5}
\end{eqnarray}for a positive constant $C_{1}>0$. Hence, the convolution method
is also stable for the approximate gradient $\dot{u}_{i}$. 

The result of equation (\ref{eq:GDR}) follows by taking the supremum
on the left hand sides of equations (\ref{eq:prf4}) and (\ref{eq:prf5})
other time steps and using Corollary \ref{cor:linext}.
\end{proof}

\section{Extensions} \label{sec:extensions}

Various extensions of the convolution method can be made.  We  consider the convolution method under decoupled FBSDEs and reflected FBSDEs. These cases have interesting applications in mathematical finance, especially for option pricing.

\subsection{Forward-backward stochastic differential equations}

We can extend the convolution method to consider FBSDEs
\begin{equation}
\begin{cases} dX_t = a(t,X_t)dt + \sigma(t,X_t)dW_t \\
-dY_t = f(t,X_t,Y_t,Z_t)dt - Z_tdW_t \\ 
X_0 = x_0 \text{ , } Y_T = g(X_T) \end{cases}
\end{equation}
associated to the Cauchy problem for the advection-diffusion
equation 
\begin{equation} 
\begin{cases} 
\Done{u}{t} + a(t,x) \Done{u}{x} + \frac{1}{2} \sigma^2(t,x) \Dtone{u}{x} + f(t,x,u, \sigma(t,x) \Done{ u}{x}) = 0 \text{ , } (t,x) \in [0,T) \times \mathbb{R}  \\
u(T,x) = g(x) \text{, } x \in \mathbb{R} 
\end{cases}
\end{equation}
to which an obstacle can be added when in presence of a reflected FBSDE. When discretized with the Euler scheme, the FBSDE numerical solution is given by
\begin{equation} 
\begin{cases} 
{Z}^{\pi}_{t_n} = 0 \text{, } {Y}^{\pi}_{t_n} = \xi^{\pi} \text{, } {X}^{\pi}_0 = x_0 \\ 
{X}^{\pi}_{t_{i+1}} =  {X}^{\pi}_{t_i} + a(t_i, {X}^{\pi}_{t_i}) \Delta_i + \sigma(t_i, {X}^{\pi}_{t_i}) \Delta W_i \\ 
{Z}^{\pi}_{t_i} = \frac{1}{\Delta_i} \CEsp{}{ {Y}^{\pi}_{t_{i+1}} \Delta W_i | \mathcal{F}_{t_i} }{} \\
 {Y}^{\pi}_{t_i}= \CEsp{}{{Y}^{\pi}_{t_{i+1}} | \mathcal{F}_{t_i}}{}  +   f({t_i},{X}^{\pi}_{t_i}, \CEsp{}{{Y}^{\pi}_{t_{i+1}} | \mathcal{F}_{t_i}}{} ,{Z}^{\pi}_{t_i}) .
\end{cases} \label{eq:convschemeR11} 
\end{equation}
The approximate solutions satisfy\begin{eqnarray} u_i(x)  & = &  \tilde{u}_i(x) + \Delta_i f(t_i, x, \tilde{u}_i(x) , \dot{u}_i(x)) \label{eq:asfbsde}\\
\dot{ u}_i (x) & = & \frac{ \sigma^{-1}(t_i, x) }{\Delta_i  } \int_{\mathbb{R}}^{} (y-\Delta_ia(t_i,x)) u_{i+1}(x + y) h_i(y|x) dy \nonumber \\
 & = & e^{-\alpha x}\sigma(t_i,x) \mathfrak{F}^{-1}[ \mathfrak{F}[u^{\alpha}_{i+1}](\nu) (\alpha + {\bf i} \nu) \phi_i(\nu - {\bf i} \alpha,x) ](x) \label{eq:agfbsde}
\end{eqnarray} where \begin{eqnarray} \tilde{u}_i(x) & = & \int_{\mathbb{R}}^{} u_{i+1}(x+y) h_i(y|x) dy \nonumber \\
 & = & e^{-\alpha x} \mathfrak{F}^{-1}[ \mathfrak{F}[u^{\alpha}_{i+1}](\nu)  \phi_i(\nu- {\bf i} \alpha,x) ](x) \label{eq:isfbsde} 
\end{eqnarray}  for $i=0,1,...,n-1$ and $u_n(x) = g(x)$.
The density of increments is Gaussian
\begin{equation} 
h_i(y|x) = \frac{1}{( 2 \pi \Delta_i )^{\frac{1}{2}} \sigma} \exp \left( -\frac{ (y - a(t_i,x) \Delta_i)^2 }{2 \sigma^2(t_i,x) \Delta_i} \right) 
\end{equation}
with characteristic function 
\begin{equation} 
\phi_i(\nu,x) = \exp \left( {\Delta_i\left( {\bf i} a(t_i,x)  \nu - \frac{1}{2} \sigma^2(t_i,x) \nu^2 \right)} \right). 
\end{equation}
The development of the convolution method in this case also leads to transforms identical to equation (\ref{eq:Pint}).
In our implementation, the $\dot{u}$ and $\dot{v}$
are actually estimates for $\sigma(t,x)\Done{u}{x}$
but the scheme can easily be modified so as to estimate the derivative
$\Done{ u}{x}$ directly. 

The equivalence \begin{eqnarray} \theta (x) & = & \frac{1}{2\pi} \int_{-\infty}^{\infty} e^{ {\bf i} \nu x } \hat{\eta^{\alpha}}(\nu) \psi( \nu ) d\nu \nonumber \\ & = &\frac{1}{2\pi} \int_{-\infty}^{\infty} e^{ {\bf i} \nu x } \widehat{\eta^{\alpha}_{ \beta, \kappa}} (\nu) \psi( \nu ) d\nu - H(x,\alpha, \beta, \kappa) \end{eqnarray} 
of Theorem \ref{thm:adjust} still holds with
\begin{equation}   
H(x,\alpha, \beta, \kappa) =  e^{-\alpha x}( \beta (x + a(t_i,x) \Delta_i) + \kappa) 
\end{equation}  if $\psi(\nu) = \phi_i(\nu - {\bf i}\alpha,x)$ and 
\begin{equation}
 H(x,\alpha, \beta, \kappa) = e^{-\alpha x} \beta \sigma(t_i,x) 
\end{equation}
if $\psi(\nu) = \sigma(t_i,x) (\alpha + {\bf i} \nu) \phi_i(\nu - {\bf i}\alpha,x)$.
 Whenever the forward coefficients $a$ and $\sigma$ depend on the state variable $x$, a matrix multiplication in required to perform the DFTs. 
 
Also, the convolution method can be used to compute conditional expectations under general L\'evy processes as in \cite{lordal:2008} \revised{which focuses on American option pricing under L\'evy processes}. Indeed, the independence of increments and the availability of the characteristic function are the only requirements to apply the method.  The convolution method may also serve as a numerical method for partial differential integral equations (PIDE) under a L\'evy process \revised{similar to \cite{lordal:2008}}.

\subsection{Reflected FBSDEs}

Explicit Euler schemes have been constructed for reflected FBSDEs with continuous barrier which make it possible to apply the convolution method to such FBSDEs. Consider the solution
$(X,Y,Z,A)$ of the system 
\begin{equation} 
\begin{cases} dX_t = a(t,X_t)dt + \sigma(t,X_t)dW_t \\ -dY_t = f(t,X_t,Y_t,Z_t,)dt +  dA_t - Z_tdW_t \\ 
Y_t \geq B_t \text{ , } dA_t \geq 0  \text{ , } \forall t \in [0,T] \\
\int_0^T (Y_t - B_t)dA_t = 0  \text{ , } X_0 = x_0 \text{ , } Y_T = g(X_T) \end{cases}
\end{equation}
where the lower barrier is a deterministic function $B:[0,T]\times\mathbb{R}\rightarrow\mathbb{R}$
of time and the Brownian motion and 
\begin{equation} 
B_t = B(t,X_t). 
\end{equation}
This reflected FBSDE is associated to the following obstacle problem 
\begin{equation} 
\begin{cases} 
\Done{u}{t} + a(t,x) \Done{u}{x} + \frac{1}{2} \sigma^2(t,x) \Dtone{u}{x} + f(t,x,u, \sigma(t,x) \Done{ u}{x}) = 0 \text{ , } (t,x) \in [0,T) \times \mathbb{R}  \\
u(t,x) \geq B(t,x) \text{, } (t,x) \in [0,T] \times \mathbb{R}  \\
u(T,x) = g(x) \text{, } x \in \mathbb{R} 
\end{cases} 
\end{equation}
as established by \cite{elkarouial:1997}. An adaptation of the explicit Euler scheme I provides the numerical solution to the reflected BSDE
through the equations
\begin{equation} 
\begin{cases} 
{Z}^{\pi}_{t_n} = 0 \text{, } {Y}^{\pi}_{t_n} = \xi^{\pi} \text{, } {X}^{\pi}_0 = x_0 \\ 
{X}^{\pi}_{t_{i+1}} =  {X}^{\pi}_{t_i} + a(t_i, {X}^{\pi}_{t_i}) \Delta_i + \sigma(t_i, {X}^{\pi}_{t_i}) \Delta W_i \\ 
{Z}^{\pi}_{t_i} = \frac{1}{\Delta_i} \CEsp{}{ {Y}^{\pi}_{t_{i+1}} \Delta W_i | \mathcal{F}_{t_i} }{} \\
\Delta A^{\pi}_{t_i} = \left ( \CEsp{}{{Y}^{\pi}_{t_{i+1}}+f({t_i},{X}^{\pi}_{t_i},{Y}^{\pi}_{t_{i+1}},{Z}^{\pi}_{t_i})\Delta_i | \mathcal{F}_{t_i}}{} - B(t_i, W_{t_i})      \right)^{-} \\
 {Y}^{\pi}_{t_i}= \CEsp{}{{Y}^{\pi}_{t_{i+1}}+f({t_i},{X}^{\pi}_{t_i},{Y}^{\pi}_{t_{i+1}},{Z}^{\pi}_{t_i})\Delta_i | \mathcal{F}_{t_i}}{} + \Delta A^{\pi}_{t_i} 
\end{cases} \label{eq:convschemeR1} 
\end{equation}
where for any number $x \in \mathbb{R}$, $x^- = \max(0,-x)$.

 Time discretization of RBSDEs and their convergence were treated in \cite{bouchardchass:2008} for the implicit Euler scheme. \cite{pengxu:2011} proposed an equivalent scheme with a discrete filtration and proved its convergence under a binomial method. The scheme is easily solved with a convolution method by first noticing that the approximate solution $(v_{i}, \dot{v}_{i}, \Delta \bar{v}_{i})$ at mesh time $t_i$, where  $\Delta\bar{v}_{i}$ is the approximate reflection process,  can be written as
\begin{eqnarray} 
v_i(x) & = &  \int_{-\infty}^{\infty} \tilde{v}_{i+1}(x+y) h_i(y|x) dy  + \Delta \bar{v}_i(x) \nonumber\\
       & = & e^{-\alpha x} \mathfrak{F}^{-1}[ \mathfrak{F}[\tilde{v}^{\alpha}_{i+1}](\nu)  \phi_i(\nu- {\bf i} \alpha,x) ](x)  + \Delta \bar{v}_i(x) 
\end{eqnarray}
where 
\begin{eqnarray} \tilde{v}_{i+1}(x)  & = & v_{i+1}(x) + \Delta_i f(t_i, x, v_{i+1} (x), \dot{v}_i(x) ) \text{, } \\
\dot{v}_i(x) & = & \frac{\sigma^{-1}(t_i,x)}{\Delta_i} \int_{-\infty}^{\infty}(y-a(t_i,x)) v_{i+1}(x+y) h(y |x) dy \nonumber \\ 
& = &  e^{-\alpha x}\sigma(t_i,x) \mathfrak{F}^{-1}[ \mathfrak{F}[v^{\alpha}_{i+1}](\nu) (\alpha + {\bf i} \nu) \phi_i(\nu - {\bf i} \alpha,x) ](x)\\
\Delta \bar{v}_i(x) & = & \left( e^{-\alpha x} \mathfrak{F}^{-1}[ \mathfrak{F}[\tilde{v}^{\alpha}_{i+1}](\nu)  \phi_i(\nu- {\bf i} \alpha,x) ](x)  - B(t_i,x) \right)^{-} 
\end{eqnarray}
for $i=0,1,...,n-1$ and $v_{n}(x)=g(x)$.  The computation of $\dot{v}_{i}$ and the integral part of the approximate solution ${v}_{i}$ is identical to the non-reflected case.

One can also naturally build an alternative scheme from the explicit
Euler scheme II. The approximate solution $u_{i}$, the approximate gradient
$\dot{u}_{i}$ and the approximate reflection $\Delta \bar{u}_{i}$ at mesh
time  $t_{i}$ then take the form 
\begin{eqnarray}
 u_i(x)  & = &  \tilde{u}_i(x) + \Delta_i f(t_i, x,\tilde{u}_i(x) , \dot{ u}_i(x)) + \Delta \bar{u}_i (x) \label{eq:ufR} 
\end{eqnarray}
where
\begin{eqnarray} 
\dot{ u}_i (x) & = & \frac{ \sigma^{-1}(t_i, x) }{\Delta_i  } \int_{\mathbb{R}}^{} (y-\Delta_ia(t_i,x)) u_{i+1}(x + y) h_i(y|x) dy \nonumber \\
 & = & e^{-\alpha x}\sigma(t_i,x) \mathfrak{F}^{-1}[ \mathfrak{F}[u^{\alpha}_{i+1}](\nu) (\alpha + {\bf i} \nu) \phi_i(\nu - {\bf i} \alpha,x) ](x) \label{eq:gufR} \\
\tilde{u}_i(x) & = & \int_{\mathbb{R}}^{} u_{i+1}(x+y) h_i(y|x) dy \nonumber \\
 & = & e^{-\alpha x} \mathfrak{F}^{-1}[ \mathfrak{F}[u^{\alpha}_{i+1}](\nu)  \phi_i(\nu- {\bf i} \alpha,x) ](x)\\
 \Delta \bar{u}_i(x) & = & \left( \tilde{u}_i(x) + \Delta_i f(t_i,x, \tilde{u}_i(x) , \bar{ u}_i(x)) - B(t_i,x) \right)^{-}
\end{eqnarray}
for $i=0,1,...,n-1$ and $u_n(x) = g(x)$.

\section{Application to option pricing}  \label{sec:numres}
We shall consider the case of BSDEs with non-linear and non-smooth drivers that present the lowest rate of convergence through option pricing problems.  An introduction to financial applications of BSDEs, particularly
to imperfect markets and American option problems, can be found in \cite{elkarouiq:1997}, \cite{elkarouietal:1997} or \cite{elkaouri:1996}.

The market model consists of a single risky asset (or stock)
${S}$ with the dynamics \begin{equation} S_t = e^{X_t} \end{equation}
where the process ${X}$ represents the stock return. We first consider
a European call option with maturity $T$ and strike price $K$
under a lending rate of $r$ and a borrowing rate $R$. The return
process is an arithmetic Brownian motion \begin{equation} X_t = X_0 + \left(\mu - \delta - \frac{1}{2} \sigma^2\right) t + \sigma W_t \end{equation}
such that the stock has an initial value of $S_{0}=e^{X_{0}}$,
an expected return rate of $\mu$, a dividend rate of $\delta$ and a volatility of $\sigma$. 

The European call option price follows a BSDE with the return process
${X}$ as the forward process, the driver \begin{equation} f(t,y,z) = -ry - \left( \frac{ \mu - r}{\sigma} \right)z + (R-r)\left( y - \frac{z}{\sigma}\right)^{-} \end{equation}and
the terminal function \begin{equation} g(x) = \left( e^x - K \right)^{+} \end{equation}
under the given imperfect market conditions. The American call option solves a reflected BSDE with the
barrier function \begin{equation} B(t,x) = g(x) = (e^x - K)^+ \text{ , }  (t,x) \in [0,T]\times \mathbb{R}. \end{equation}

\subsection{Numerical Results}
Suppose that $T=1$, $S_{0}=100$, $r=0.01$,  $\mu=0.05$, and $\sigma=0.20$.
When the borrowing rate equals the lending rate $R=r=0.01$ and $\delta=0$, the European and the American call options have the same price. Figure \ref{fig:Error} shows the structure of the absolute log error on European option prices and deltas where the true values are computed using the Black-Scholes formula. As expected errors are amplified at the boundaries of the truncated domain.   Errors  are also amplified for around-the-money options,  to a lesser extent, due to the non-smoothness of the terminal function $g$.   In addition, out-of-the-money options have smaller absolute errors compared to in-the-money options.

\begin{figure}
\begin{center}
\caption{Absolute error of the convolution method (scheme II) for \revised{European} call option prices and deltas (non-dividend paying stock and no market frictions)} 
\resizebox{\textwidth}{!}
{%
\includegraphics[bb=0 0 4667 3501]{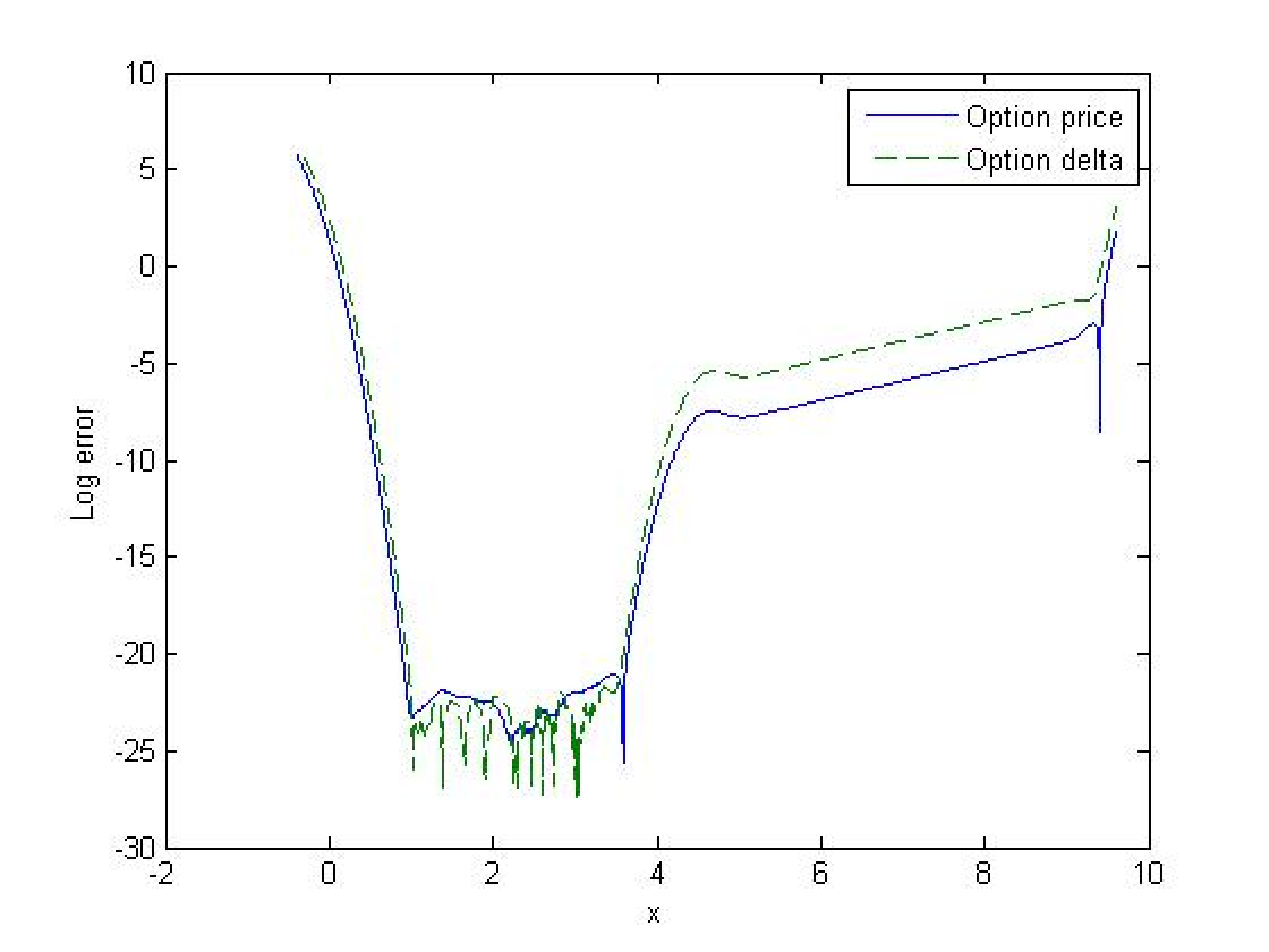}
}
\label{fig:Error}
{ $[x_0,x_N] = X_0 + [-5,5]$, $N=2^{12}$, $n=1000$, $\epsilon = 5$, $K = S_0 = 100$, $R=r=0.01$, $\delta = 0$.}
\end{center}
\end{figure}

Table~\ref{tab:callpm} gives the 
relative percentage error for price estimates of the \revised{European} call option prices using both convolution
schemes I and II  for different time steps and
the indicated strike prices. Table \ref{tab:calldm} gives the
estimates for the \revised{European} call option deltas obtained from the approximate gradient
by 
\begin{equation}  
\text{Delta} = \frac{ \dot{u}_0(X_0) }{ \sigma S_0 }  
\end{equation}
when using the explicit Euler scheme II. The option prices and option deltas are calculated on the restricted domain $[x_0,x_N] = X_0 + [-5,5]$ with $N = 2^{12}$ grid points and a minimal slope of $\epsilon = 5$ (see Remark~\ref{rem:slope}). The Black-Scholes formula gives call option prices of $4.6101$, $8.4333$ and $14.1929$
at strike prices $K=110$, $100$ and $90$ respectively when the other variables are kept unchanged. The true
values for the option deltas are $0.3720$, $0.5596$ and $0.7507$
when the strike price is $K=110,100$ and $90$ respectively.

The results of Table \ref{tab:callpm} and \ref{tab:calldm} show the accuracy of the convolution method on a RBSDE with a smooth linear driver. Indeed, the relative error percentages remain low (less than $0.3\%$) for the estimated option prices and deltas.  Out-of-the-money option estimates seem to display the largest relative errors.   

\begin{table}
\begin{center} 
\caption{Relative errors (\%) for \revised{European} call option prices (non-dividend paying stock and no market frictions)} 
\scalebox{1.0}{
\begin{tabular}{l c c c c c} 
\\
\hline
 & K (Strike) & n=500 & n=1000 & n=2000 & n=5000 \\ \hline \hline
 \multirow{3}{*}{ Convolution (Scheme I) } & 110 & 0.0456 & 0.0217 & 0.0108 & 0.0043 \\
                                         & 100 & 0.0178 &   0.0095 & 0.0047 & 0.0024 \\ 
                                         & 90 & 0.0049 & 0.0028 & 0.0014 & 0.0007 \\ \hline 
\multirow{3}{*}{ Convolution (Scheme II) } & 110 & 0.0087 & 0.0239 & 0.0022 & 0.0001 \\
                                         & 100 & 0.0059 & 0.0024 & 0.0012 & 0.0007 \\
                                         & 90  & 0.0028 & 0.0014 & 0.0007 & 0.0004 \\   \hline   \\
\end{tabular} 
}
\label{tab:callpm} 

{ $[x_0,x_N] = X_0 + [-5,5]$, $N = 2^{12}$, $\epsilon = 5$, $R=r=0.01$, $\delta=0$.}
\end{center} 
\end{table}

\begin{table}
\begin{center} 
\caption{Relative errors (\%) for \revised{European} call option deltas (non-dividend paying stock and no market frictions)}
\scalebox{1.0}{ 
\begin{tabular}{l c c c c} 
\\
\hline
K (Strike) & 90 & 100 & 110\\ \hline \hline
Convolution (Scheme I) & 0.0133 &  0.0010 & 0.2414 \\
Convolution (Scheme II) & 0.0133 & 0.0010 & 0.2414 \\       
\hline \\
\end{tabular} }
\label{tab:calldm} 

{$[x_0,x_N] = X_0 + [-5,5]$, $N = 2^{12}$, $\epsilon = 5$, $R=r=0.01$, $\delta=0$.}
\end{center} 
\end{table}

\revised{When the borrowing and lending rates differ, in the absence of dividends, the European and American call option prices agree and we may employ the Black-Scholes formula with the higher borrowing rate to price the call option (see \cite{gobetLecNotes,MR1233619,korn95,bergman95}).
Nevertheless, we employ the convolution method without these simplifications in order to test the method.} For a borrowing rate of $R=0.03$ and a dividend rate of $\delta=0$, Table \ref{tab:callpim} shows
the estimates for the European call prices when the option is at the
money $S_0=K=100$ and $r=0.01$. \revised{The values in Table \ref{tab:callpim} are clearly converging to the true value, 9.4134, calculated using the Black-Scholes formula with borrowing rate $R$.}  Moreover, the convolution methods return an option delta of $0.5987$, when applied with $n=2000$ time steps\revised{, which agrees with the delta calculated from the Black-Scholes formula with the higher borrowing rate}. \revised{Note that the simulated paths  in Figure \ref{fig:simAm}} never hit the \revised{early exercise} barrier \revised{as expected}.  Paths are simulated using the solution from the convolution method applied on scheme II on the restricted domain $[x_0,x_N] = X_0 + [-5,5]$ with $N=2^{12}$ grid points, $n=1000$ time steps and $\epsilon = 5$. We used $n=1000$ time steps to simulate the stock price ($S_t$).

\begin{table}
\begin{center} 
\caption{ATM \revised{European} call option prices  (non-dividend paying stock under imperfect market conditions)} 
\scalebox{1.0}{
\begin{tabular}{c c c c c} 
\\
\hline
n (number of time steps) & 500 & 1000 & 2000 & 5000 \\
Convolution (Scheme I) & 9.4127 &   9.4131 & 9.4132 & 9.4133 \\
Convolution (Scheme II) & 9.4132 & 9.4133 & 9.4133 & 9.4134 \\  
\hline \\
\end{tabular} 
}
\label{tab:callpim}

{ $[x_0,x_N] = X_0 + [-5,5]$, $N=2^{12}$, $\epsilon = 5$, $R=0.03$, $r=0.01$, $K=S_0=100$, $\delta=0$.}

\end{center} 
\end{table}

\begin{figure}
\begin{center}
\caption{(ATM) \revised{European} call option sample paths  (non-dividend paying stock under imperfect market conditions)} 
\resizebox{\textwidth}{!}{
\includegraphics[bb=0 0 4667 3501,clip]{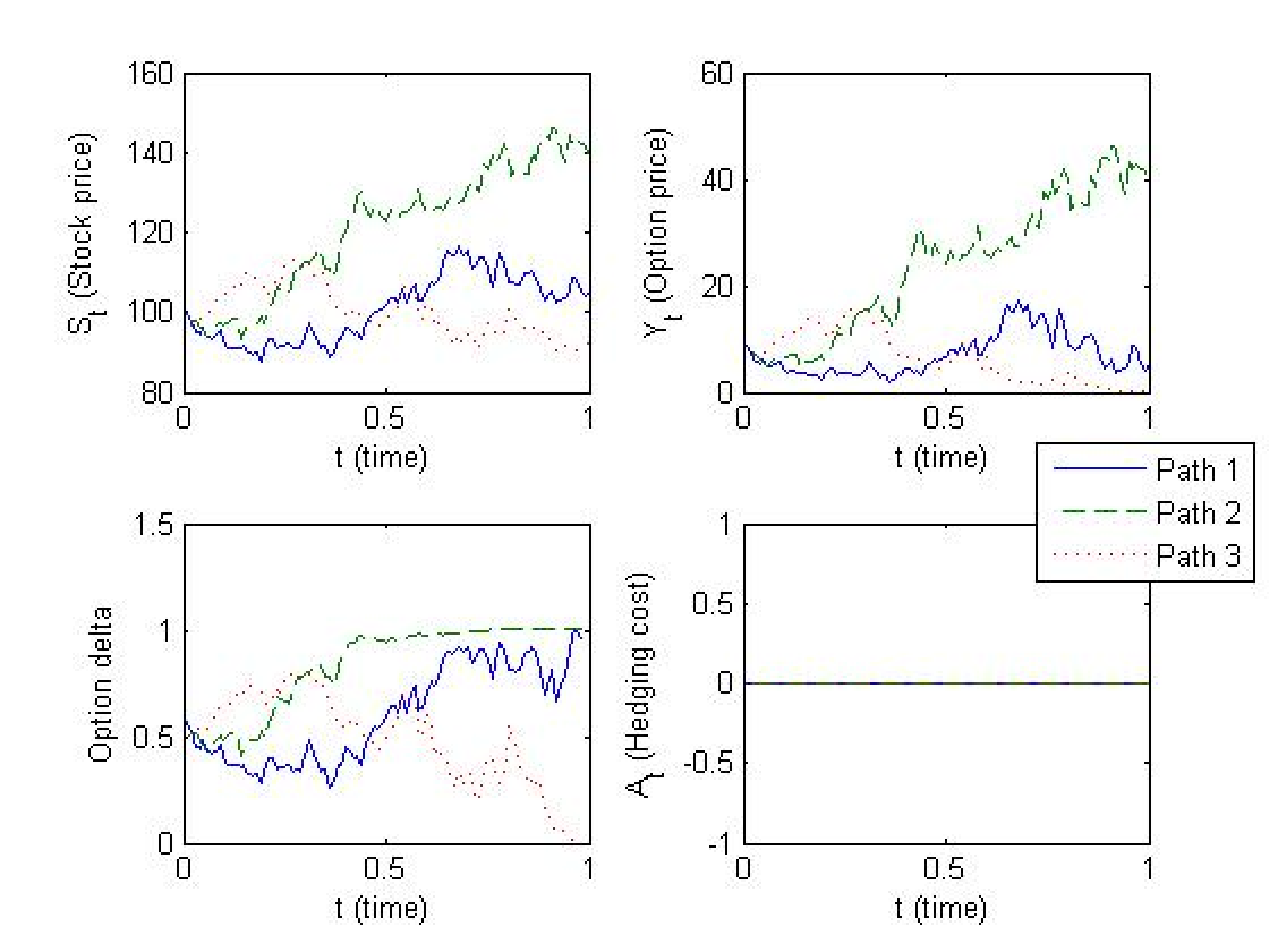}
}

\label{fig:simAm}

{$[x_0,x_N] = X_0 + [-5,5]$, $N=2^{12}$, $n=1000$, $\epsilon = 5$, $K = S_0 = 100$, $R=0.03$, $r=0.01$, $\delta = 0$.}

\end{center}
\end{figure}

The option price can be calculated using a Monte-Carlo method such as
the forward scheme of \cite{benderdenk:2007}.  However,
in the context of uni-dimensional BSDEs, Monte-Carlo methods will generally
be less efficient than space discretization methods. As an illustration,
the convolution method on both explicit Euler schemes runs in approximately
$4$ seconds when pricing the options of Table \ref{tab:callpim}
with $n=1000$ time steps. On the other hand, the forward scheme runs in $18$
seconds with only $n=20$ time steps. We used the $7$ first power functions as basis functions and $100,000$ paths to
generate the Monte-Carlo estimates. The Picard iterations are stopped whenever
the difference in two consecutive prices is less than $10^{-4}$ for
a maximum number of $10$ iterations. Fifty $(50)$ independent valuations with the Monte-Carlo method
give a $95\%$ confidence interval of $[9.3972,9.4222]$ which includes
all estimates of Table \ref{tab:callpim}. Hence, the convolution method gives satisfactory results even for coarse time discretization. 

Table \ref{tab:callpim2} provides price estimates for out-of-the-money and in-the-money options
and Table \ref{tab:calldm2} gives estimates for option deltas on a non-dividend-paying stock under imperfect market conditions. 
Both tables compare the estimates obtained with the convolution method \revised{to} those obtained with the binomial method of \cite{pengxu:2011} \revised{and the true values calculated using the Black-Scholes formula with the higher borrowing rate $R$}. The convolution method and the binomial method give similar prices and delta values for all options which confirms the convolution method accuracy even for non-smooth drivers. Nonetheless, the binomial method is faster (less that tenth of a second for $1000$ time steps) than the convolution method for the same number of time step when computing the BSDE initial values. However, simulation of BSDEs is easier under the convolution method where a simple interpolation using Brownian paths can be performed. Under the binomial method, Brownian paths must be approximated by scaled random walks.

\begin{table}
\begin{center} 
\caption{\revised{European} call option prices (non-dividend paying stock under imperfect market conditions)} 
\scalebox{1.0}{
\begin{tabular}{l c c c c c } 
\\
\hline
n (number of time steps) & K (Strike) & 500 & 1000 & 2000 & 5000  \\ \hline \hline
\multirow{2}{*}{ Convolution (Scheme I) } & 110 & 5.2924 & 5.2929 & 5.2931 & 5.2933 \\
                                         & 90 & 15.4289 & 15.4291 & 15.4292 & 15.4292 \\ \hline 
\multirow{2}{*}{ Convolution (Scheme II) } & 110 & 5.2932 & 5.2933 & 5.2933 & 5.2934 \\
                                          & 90  & 15.4290 & 15.4291 & 15.4291 & 15.4292 \\ \hline
\multirow{2}{*}{ Binomial Method } & 110 & 5.2918 & 5.2945 & 5.2936 & 5.2937 \\
                                          & 90  & 15.4313 & 15.4263 & 15.4298 & 15.4292 \\ \hline 
\multirow{2}{*}{ \revised{ Black-Scholes (exact)} } & \revised{110} & \multicolumn{4}{c} {\revised{5.2933}} \\ 
 & \revised{90} & \multicolumn{4}{c} {\revised{15.4292}} \\ \hline
\end{tabular} }
\label{tab:callpim2} 

{ $[x_0,x_N] = X_0 + [-5,5]$, $N=2^{12}$, $\epsilon = 5$, $R=0.03$, $r=0.01$, $\delta=0$.}

\end{center} 
\end{table}

\begin{table}
\begin{center} 
\caption{\revised{European} call option deltas (non-dividend paying stock under imperfect market conditions)} 
\scalebox{1.0}{
\begin{tabular}{l c c c c} 
\\
\hline
K (Strike) & 90 & 100 & 110\\ \hline \hline
Convolution (Scheme I) & 0.7814 &   0.5987 & 0.4104 \\
Convolution (Scheme II) & 0.7814 &  0.5987 & 0.4104 \\ 
Binomial Method       & 0.7813 & 0.5987  & 0.4104 \\      
\revised{Black-Scholes (exact)}         & \revised{0.7814} & \revised{0.5987} & \revised{0.4104} \\ 
\hline \\
\end{tabular} }
\label{tab:calldm2} 

{ $[x_0,x_N] = X_0 + [-5,5]$, $N = 2^{12}$, $\epsilon = 5$, $n=2000$, $R=0.03$, $r=0.01$, $\delta = 0$.}

\end{center} 
\end{table}

\begin{figure}
\begin{center}
\caption{(ATM) American call option sample paths (dividend-paying stock under imperfect market conditions)} 
\resizebox{\textwidth}{!}{
\includegraphics[bb=0 0 4667 3501,clip]{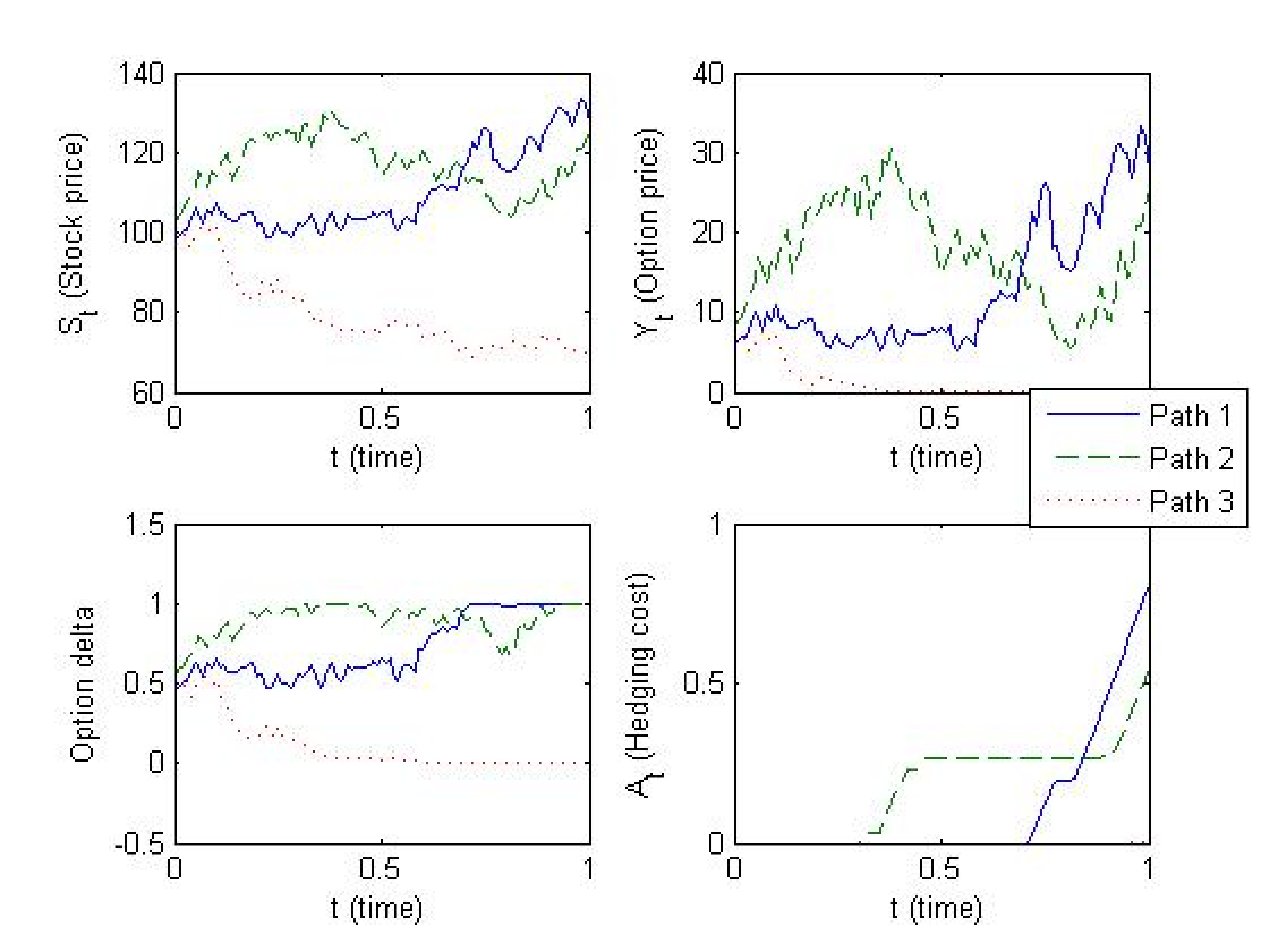} 
}
\label{fig:simAm2}

{ $[x_0,x_N] = X_0 + [-5,5]$, $N=2^{12}$, $n=1000$, $\epsilon = 5$, $K = S_0 = 100$, $R=0.03$, $r=0.01$, $\delta = 0.035$.}

\end{center}

\end{figure}

\begin{figure}

\begin{center}
\caption{(ATM) American call option price and delta surfaces \newline (dividend paying stock with market frictions)} 
\resizebox{\textwidth}{!}{
\includegraphics[bb=0 0 4667 3501,clip]{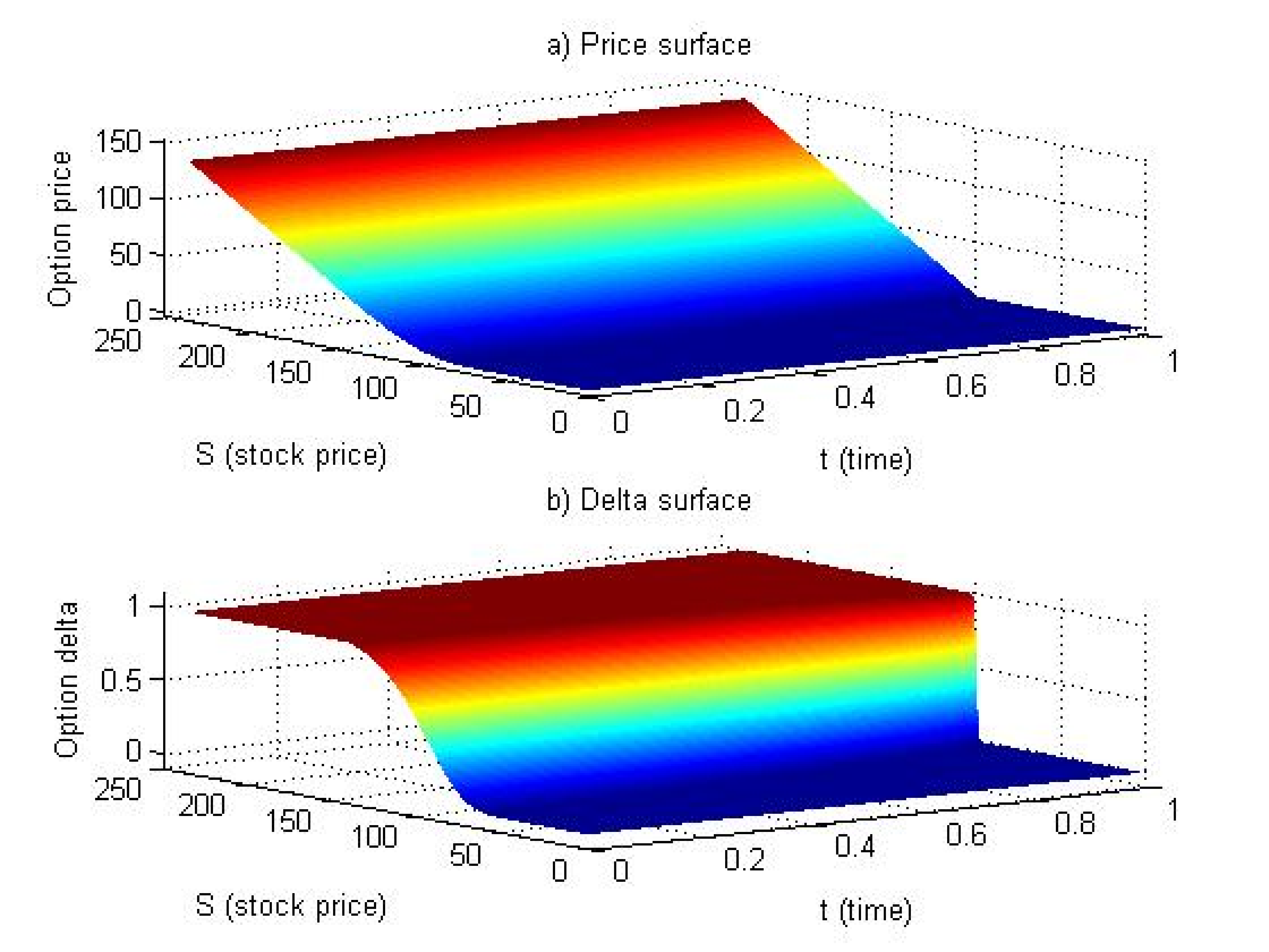} 
}
\label{fig:solsimAm}

{$[x_0,x_N] = X_0 + [-5,5]$, $N=2^{12}$, $n=1000$, $\epsilon = 5$, $K=S_0 = 100$, $R=0.03$, $r=0.01$, $\delta=0.035$.}

\end{center}
\end{figure}

If we introduce a dividend rate of $\delta=0.035$ under imperfect market conditions ($R=0.03$ and $r=0.01$), then the American and the European call option prices differ and the Black-Scholes formula does not apply. 
The convolution method estimates the (at-the-money) American call option price at $7.5610$
and the European call option price at $7.4712$. We use scheme II with the restricted domain $[x_0,x_N] = X_0 + [-5,5]$,
$N=2^{12}$ grid points, $n=2000$ time steps and a minimal slope
of $\epsilon = 5$. Figure \ref{fig:simAm2} shows the typical sample paths for the American option where the reflecting process $A_t$ (hedging cost) is now non-zero for in-the-money path indicating a difference in price with the European call option.

Figure \ref{fig:solsimAm} displays the option price and delta surfaces. The regularity of these surfaces indicates that the convolution method is efficient in handling non-smoothness in the terminal condition $g$ but also in the driver $f$.

\section{Conclusion} \label{sec:concl} In this paper we presented  a new spatial discretization method for the numerical solution of backward stochastic differential equations (BSDEs).  This new method expresses conditional expectations appearing in explicit Euler time discretizations of the BSDE as convolutions.   The convolution theorem of Fourier analysis is then applied in order to derive a recursive, backward in time, method for computing the numerical solution of the BSDE in terms of inverse Fourier transforms of previous time step solutions.   After discretizing the state and Fourier space these expressions can be implemented using the fast Fourier transform (FFT) algorithm. Since the FFT algorithm is more suitable for periodic functions we introduced a transform in order to treat BSDEs with non-periodic terminal conditions.  A (local) error analysis is provided which indicates that the use of the FFT performs an extrapolation that induces a non-negligible error term meaning that transform does not completely solve the problem of non-periodicity.  We extend the convolution method to consider forward-backward stochastic differential equation (FBSDEs) and reflected FBSDEs which are an important extensions for financial applications.

Numerical experiments, in the context of option pricing problems, show that the convolution method is accurate and handles non-linearity and non-smoothness in the BSDE coefficients.  The addition of a technique to suppress the extrapolation error is an interesting improvement to the method and shall be presented in a future paper (see \cite{poly:PhD}).  \revised{Application of the convolution method to other examples; such as pricing call-spread options under different borrowing and lending rates as in \cite{MR2265667}, \cite{MR3287774}, \cite{RuijOosl:2013}; can provide further opportunities to study the effectiveness of the method.}   Efficient implementations for multidimensional problems is an important area of future research.

\bibliographystyle{siam}
\bibliography{biblio-rev}

\end{document}